\documentclass[12pt]{article}

\usepackage[a4paper,margin=30mm]{geometry}
\usepackage{amsmath,amssymb,amsthm,mathtools}
\usepackage[T1]{fontenc}
\usepackage{lmodern}
\usepackage{microtype}
\usepackage[hidelinks]{hyperref}
\usepackage{enumerate}
\usepackage{xcolor}
\usepackage{etoolbox}

\newtheorem{theorem}{Theorem}[section]
\newtheorem{lemma}[theorem]{Lemma}
\newtheorem{proposition}[theorem]{Proposition}
\newtheorem{corollary}[theorem]{Corollary}
\newtheorem{definition}[theorem]{Definition}
\theoremstyle{remark}
\newtheorem{remark}{Remark}

\numberwithin{equation}{section}

\newcommand{\E}{\mathbb E}
\newcommand{\Pp}{\mathbb P}
\newcommand{\one}{\mathbf 1}
\newcommand{\mesh}{\operatorname{mesh}}
\newcommand{\tr}{\operatorname{tr}}
\newcommand{\HS}{\mathrm{HS}}

\usepackage[disable]{todonotes}
\newcommand{\purba}[2][]{\todo[color=yellow!60!white, size=\footnotesize, #1]{PD: #2}}

\title{Quadratic and $p$-th variation of random signed\\
Takagi--Landsberg bridges}
\author{Purba Das\footnote{Department of Mathematics, King's College London. E-mail: purba.das@kcl.ac.uk} \, and \, Alexander Schied\footnote{Department of Statistics and Actuarial Science, University of Waterloo. E-mail: aschied@uwaterloo.ca}}
\date{September 18, 2026}

\begin{document}
\maketitle

\begin{abstract}
We study random signed Takagi--Landsberg bridges with independent
Rademacher Faber--Schauder coefficients. At index $H=1/2$, we prove that
every fixed deterministic refining sequence of partitions with vanishing
mesh yields quadratic variation $t$, almost surely and uniformly in
$t\in[0,1]$. For partitions that need not be refining, the mesh condition
$o(1/\log n)$ is sufficient, and its order is sharp. Our proofs use an
operator representation of the quadratic sums and moment bounds for
Rademacher chaos. We also show that asymmetric signs can destroy this
invariance. At every index $H\ne1/2$, we construct a deterministic refining
sequence along which the critical $p$-power sums, with $p=1/H$, have two
distinct finite accumulation points almost surely. At $H=1/4$, uniform
thirds grids provide an explicit alternative to the dyadic limit.
Thus $H=1/2$ is the unique index in the symmetric family at which critical
variation is invariant along every fixed deterministic refining sequence. Nevertheless the variation index equals $1/H$ almost surely across a large sequence of partitions.
\end{abstract}

\section{Introduction}

Paul L\'evy~\cite{Levy1940} proved that Brownian motion has linear quadratic variation along every fixed deterministic
refining sequence of partitions.  More precisely, let $B$ be standard Brownian motion and,
for every $n$, write
\[
  \pi_n=\{0=t^n_0<t^n_1<\cdots<t^n_{N_n}=T\},\qquad n\in\mathbb N,
\]
for partitions of the time interval $[0,T]$. If $(\pi_n)$ is any fixed deterministic refining sequence with vanishing
mesh, then
\[
  \sum_{i=0}^{N_n-1}(B_{t^n_{i+1}}-B_{t^n_i})^2
  \longrightarrow T
  \qquad\text{almost surely}.
\]
Applying the terminal-time statement at rational times and using
monotonicity yields the familiar pathwise formulation $\langle B\rangle_t=t$
simultaneously for all $t\in[0,T]$.

For L\'evy's later systematic treatment of Brownian sample path properties,
see his 1948 monograph~\cite{Levy1948}.  Wiener~\cite{WienerQV} had earlier
introduced a class of functions with bounded quadratic variation in his
study of Fourier coefficients, although his notion of $2$-variation
differs from the limit along a single partition sequence used here.
Protter~\cite[Chapter~I, Theorem~28]{Protter} gives a short,
modern proof of L\'evy's theorem via the backward martingale convergence theorem.  

For Brownian motion, the condition that the partitions $(\pi_n)$ are nested can be relaxed.  
The conclusion of $L^2$-convergence 
holds whenever the mesh sizes $\mesh(\pi_n)$ converge to zero as $n\uparrow\infty$.  Dudley~\cite[p.~89]{Dudley1973} proved almost-sure
convergence for arbitrary deterministic partitions satisfying
$\mesh(\pi_n)=o(1/\log n)$.  Fern\'andez de la
Vega~\cite{DeLaVega1974} showed that this condition is best possible by constructing a partition sequence with $\mesh(\pi_n)=O(1/\log n)$ for which almost-sure convergence to the usual quadratic variation fails.

The order of the quantifiers in L\'evy's theorem is essential.  It says that
for each fixed deterministic refining sequence there is a probability-one
event on which convergence holds; it does not supply one event working for
all partition sequences simultaneously.  This distinction became
especially important after F\"ollmer~\cite{Follmer} defined
quadratic variation along a prescribed partition sequence as a property of
an individual path and derived the corresponding pathwise It\^o formula.
In that setting the limit can depend strongly on the chosen partitions,
particularly when they are selected after the path is known.  The resulting
problem is to identify conditions on the path, on the partitions, or on
their interaction that restore stability.

The possible dependence of quadratic variation on the choice of a partition sequence is extreme.  Building on observations of L\'evy
and Freedman~\cite[pp.~47--48]{Freedman}, Davis, Ob{\l}\'oj, and
Siorpaes~\cite[Theorem~7.4]{DavisOblojSiorpaes} proved that, for Brownian
motion, suitable random refining partitions can realize any prescribed
jointly measurable nondecreasing process starting from zero as a limit of
quadratic sums.  They also show in Proposition~2.3 that, for optional
partitions whose oscillation along the path tends to zero almost surely,
one can extract
a subsequence along which the canonical semimartingale quadratic variation
is recovered almost surely.  This sharp contrast with the
fixed-deterministic-sequence theorem is another reason to keep the
quantifiers explicit.

Cont and Das give conditions for stability with respect to the partition choice.  In
\cite[Definition~3.2]{ContDasRoughness},  they introduce the notion of quadratic roughness, and their Theorem~4.2 shows that if a path is
$\alpha$-H\"older and quadratically rough along a balanced partition
sequence with coarsening index $\beta\in(0,1\wedge 2\alpha) $,
then its quadratic variation along that sequence exists and agrees with its
dyadic quadratic variation.  For Brownian motion, they prove
in their Theorem~3.4 that, for each $0<\beta<1$, this
roughness property holds almost surely along any fixed balanced sequence
satisfying $(\log n)^2\mesh(\pi_n)\to0$.

In \cite{ContDas}, Cont and Das
construct nonuniform Haar
and Faber--Schauder systems associated with finitely refining partitions.
For a specific class of random Schauder expansions, they prove
that the quadratic sums along the reference
partition and every finitely refining coarsening converge to $t$ in probability.  If the reference partition and its coarsening are both
balanced and complete refining, the corresponding quadratic variations
exist and agree almost surely.  They also exhibit a deterministic Schauder series for which a balanced finitely refining coarsening
changes the limiting quadratic variation.

We study partition dependence for processes that connect Brownian bridges
with classical fractal functions. Replacing the independent standard normal
coefficients in the L\'evy--Ciesielski expansion of a Brownian bridge by
independent symmetric Rademacher variables gives a non-Gaussian process
$X$ with the same covariance, $\E[X_sX_t]=s\wedge t-st$. Its sample paths
belong to the flexible class of Takagi functions introduced by
Allaart~\cite{Allaartflexible} and studied further by Mishura and
Schied~\cite{MishuraSchied}. Unlike Brownian paths, these paths are
$1/2$-H\"older continuous. A Brownian bridge inherits L\'evy's theorem from
Brownian motion by subtracting the finite-variation path $t\mapsto tB_1$.
For the Rademacher bridge, this argument is unavailable. Nevertheless, we
prove that every fixed deterministic refining sequence with vanishing mesh
yields quadratic variation $t$, almost surely and uniformly in time.

We also examine the roles of refinement and symmetry. For partitions that
need not be refining, the mesh condition $o(1/\log n)$ remains sufficient.
At order $O(1/\log n)$, we construct one deterministic sequence along which
the terminal quadratic sums have distinct lower and upper limits for every
sign array. If the independent signs have a common nonzero mean, even
refining sequences can give a limit different from the dyadic one or fail
to give a limit at all.

Our second purpose is to study critical variation at other indices.
Mishura and Schied~\cite{MishuraSchied} proved that every signed
Takagi--Landsberg function of index $H\in(0,1)$ has linear dyadic
$p$-th variation at $p=1/H$. For every $H\ne1/2$, however, we construct a
fixed deterministic refining sequence along which the terminal $p$-power
sums have two distinct finite accumulation points almost surely. At
$H=1/4$, a separate calculation gives an explicit limit on uniform thirds
grids. These results identify both the scope and the limitations of the
quadratic invariance theorem.

The results of this paper were obtained independently of the very recent preprint by Cont~\cite{Cont2026}, which establishes intrinsic quadratic roughness for the same Rademacher Faber--Schauder process and demonstrates failure of intrinsic roughness at fourth order. His concentration estimates also imply the sufficient logarithmic mesh condition considered here. Our additional contributions include quadratic invariance along every fixed deterministic refining sequence without a mesh-rate assumption, a matching sharpness construction for the logarithmic threshold, and explicit results for biased independent signs. At every $H\ne1/2$, we construct a fixed deterministic refining sequence along which the critical sums have two distinct almost-sure accumulation points, extending the partition-dependence phenomenon beyond the fourth-order case. We also establish concentration of the critical sums around their expectations for arbitrary deterministic partitions under a mesh condition, yielding a first-moment characterization of convergence and a partition-independent variation index within this class of sequences.

Section~\ref{sec:main-results} introduces the model and states the results.
The proofs are collected in Section~\ref{sec:proofs}.

\section{Main results}
\label{sec:main-results}

We first introduce the bridges and the notions of variation used below.
We then state the quadratic invariance theorem, examine the effects of
partition mesh and asymmetric signs, and turn to critical variation at
$H\ne1/2$.

\subsection{The model and notions of variation}
\purba[inline]{Sep 1: This subsection I have updated the notations/ rewritten some part of it}
Let
\[
  e_{n,k}(t)=2^{-n/2}e_{0,0}(2^nt-k),
  \qquad
  e_{0,0}(u)=\max\{0,u\wedge(1-u)\},
\]
be the Faber--Schauder functions, where $n\ge0$ and $0\le k<2^n$.
If $(\xi_{n,k})$ are independent standard normal random variables, then the series 
\[
B^\circ_t:=\sum_{n=0}^{\infty}
       \sum_{k=0}^{2^n-1}\xi_{n,k}e_{n,k}(t)
       \]
converges uniformly in $t\in[0,1]$ almost surely, and $B^\circ$ is a standard Brownian bridge. 

Let $\{\theta_{n,k}\}$ be independent Rademacher random variables on a
probability space $(\Omega,\mathcal F,\Pp)$, so that 
\[
  \Pp(\theta_{n,k}=1)
  =\Pp(\theta_{n,k}=-1)=\frac12.
\]
For $H\in(0,1)$, we define
\begin{equation}
  X^H_t
  :=\sum_{n=0}^{\infty}2^{n(1/2-H)}
       \sum_{k=0}^{2^n-1}\theta_{n,k}e_{n,k}(t),
  \qquad 0\le t\le1.
  \label{eq:XH}
\end{equation}
We write $X:=X^{1/2}$.  At level $n$, the sum in \eqref{eq:XH} has sup norm
at most $2^{-nH-1}$; hence the series converges uniformly for every sign
array.  Every realization is therefore continuous and vanishes at both
endpoints. The functions in \eqref{eq:XH} belong to the flexible Takagi class introduced
by Allaart~\cite{Allaartflexible}. Mishura and Schied~\cite{MishuraSchied}
studied this subclass under the name signed Takagi--Landsberg functions.  In the present paper, we emphasize the fact that $X^H$ is a stochastic process and call it  the
\emph{random signed Takagi--Landsberg bridge of index $H$}.  We use
``random signed'' to mean that every tent sign is chosen independently,
while we use ``bridge'' only for the endpoint pinning; the term does not
assert that the law arises by conditioning another process.

Since the Rademacher and Gaussian coefficients have the same means and
variances, the covariance at index $1/2$ is
\begin{equation}
  \E[X_sX_t]
  =s\wedge t-st,
  \label{eq:bridge-covariance}
\end{equation}
which is the covariance kernel of a standard Brownian bridge.  Thus $X$ has
the covariance of a Brownian bridge, although its law is not Gaussian.

To compare partition sequences, we use stopped power sums and their
terminal values.

For a finite partition $\tau=\{0=t_0<\cdots<t_N=1\}$, a continuous path $x$, and $p\ge1$, we set
\[
  \mesh(\tau):=\max_{0\le i<N}(t_{i+1}-t_i)
\]
and define
\begin{align}
  [x]_{(\tau)}^{(p)}(t)
  &:=\sum_{i=0}^{N-1}
       |x(t_{i+1}\wedge t)-x(t_i\wedge t)|^p,
       \qquad 0\le t\le1,
       \label{eq:stopped-p-variation}\\
  [x]_{(\tau)}^{(p)}
  &:=[x]_{(\tau)}^{(p)}(1)
    =\sum_{i=0}^{N-1}|x(t_{i+1})-x(t_i)|^p.
       \label{eq:terminal-p-sum}
\end{align}
When $p=2$, we omit the superscript and write $[x]_{(\tau)}(t)$ and
$[x]_{(\tau)}$ for the stopped and terminal quadratic sums, respectively.
\purba[inline]{ I have replaced $V_\tau^{(p)}(x;t)$ with $[x]^{(p)}_\tau (t)$ and for $p=2$ we say $[x]_\tau (t)$. We do not need the notation $S_\tau^{(p)}(x)$ that way.}
\begin{definition}[Continuous $p$-th variation]\label{def: p-th var}
Let $\pi=(\pi_m)_{m\ge1}$ be a sequence of finite partitions with vanishing
mesh. A continuous path $x$ has \emph{continuous $p$-th variation along
$\pi$} if $[x]_{(\pi_m)}^{(p)}(t)$ converges uniformly in $t\in[0,1]$ to a
continuous function.
\end{definition}

\begin{remark}
Another convention includes the entire increment whose left endpoint does
not exceed $t$. We denote these sums by
\[
  V_\tau^{(p)}(x;t)
  :=\sum_{\substack{0\le i<N\\t_i\le t}}
       |x(t_{i+1})-x(t_i)|^p,
  \qquad 0\le t\le1,
\]
and write $V_\tau(x;t):=V_\tau^{(2)}(x;t)$ when $p=2$.
\purba[inline]{17th Sep: We work with this version and replace $[x]_{(\tau),\mathrm{succ}}^{(p)}(t)$ with $V_\tau^{(p)}(x;t)$? -- Done}
Writing $\omega_x(\delta):=\sup_{|u-v|\le\delta}|x(u)-x(v)|$, we have
\begin{equation}
\sup_{0\le t\le1}
\left|V_\tau^{(p)}(x;t)-[x]_{(\tau)}^{(p)}(t)\right| \le \omega_x(\mesh(\tau))^p.
  \label{eq:successor-comparison}
\end{equation}
Indeed, only the interval beginning at $\max(\tau\cap[0,t])$ can contribute
to the difference, and both corresponding terms lie between zero and the
right-hand side. Thus the conventions have the same limits as the mesh
vanishes. The successor sums are nondecreasing in $t$; consequently,
pointwise convergence to a continuous limit also implies uniform
convergence. We will use this observation to pass from rational times to
all times.
\end{remark}

Throughout, partition sequences are deterministic and have vanishing mesh,
unless stated otherwise. We write $\pi=(\pi_m)_{m\ge1}$, where
\[
  \pi_m=\{0=t^m_0<t^m_1<\cdots<t^m_{N_m}=1\}.
\]
In particular, the partition points do not depend on the path.

\begin{definition}[Refining partitions]
A sequence $\pi$ is \emph{refining} if $\pi_m\subseteq\pi_{m+1}$ for every
$m$. It is \emph{finitely refining} if, in addition, there is a finite
constant $M$ such that each interval of $\pi_m$ contains at most $M$ new
points of $\pi_{m+1}$, uniformly in $m$.
\label{finite.refining}
\end{definition}

We denote the dyadic sequence by $\mathbb T=(\mathbb T_m)$, where
\[
  \mathbb T_m:=\{k2^{-m}:0\le k\le2^m\}.
\]
These partitions provide our benchmark. Mishura and
Schied~\cite[corrected arXiv version, Theorem~2.1]{MishuraSchied} proved
that, for $p=1/H$ and every choice of the signs,
\begin{equation}
  [X^H]_{(\mathbb T_n)}^{(p)}(t)
  \longrightarrow t C_{\mathrm{dyad}}(H),
  \qquad
  C_{\mathrm{dyad}}(H):=2^{1-p}\E|Z_H|^p,
  \label{eq:known-dyadic-limit}
\end{equation}
uniformly in $t\in[0,1]$, where
\[
  Z_H:=\sum_{m=0}^{\infty}2^{m(H-1)}Y_m
\]
and $(Y_m)$ is an auxiliary independent Rademacher sequence. The factor
$2^{1-p}$ is included in the corrected arXiv version; it is missing from
the published theorem statement and the corresponding terminal-time
calculation. We also recover the terminal limit in
Lemma~\ref{lem:finite-template} below.

A variation index can be defined even when the critical power sums do not
converge; see, e.g., \cite{das2022thesis,HanSchied2025}.
\begin{definition}[Variation index]\label{def. variation index}
For $x\in C([0,1])$, the \emph{variation index along $\pi$} is
\begin{equation}
  p^\pi(x):=\inf\left\{p\ge1:
     \limsup_{m\to\infty}[x]_{(\pi_m)}^{(p)}=0\right\},
  \label{def. alter variation index}
\end{equation}
with $\inf\varnothing=\infty$.
\end{definition}

For $1\le q<r$, continuity gives
\[
  [x]_{(\pi_m)}^{(r)}
  \le\omega_x(\mesh(\pi_m))^{r-q}[x]_{(\pi_m)}^{(q)}.
\]
It follows that
\begin{equation}
  \limsup_{m\to\infty}[x]_{(\pi_m)}^{(q)}
  =\begin{cases}
    0,&q>p^\pi(x),\\
    \infty,&1\le q<p^\pi(x).
  \end{cases}
  \label{eq : limsup q variation}
\end{equation}
Since $C_{\mathrm{dyad}}(H)$ is finite and positive,
\eqref{eq:known-dyadic-limit} shows that $p^{\mathbb T}(X^H)=1/H$ for every
sign array. We therefore call $p=1/H$ the \emph{critical exponent}.

\subsection{Quadratic variation at index \texorpdfstring{$H=1/2$}{H=1/2}}
\purba[inline]{Sep 1: Brownian index was not defined in the paper, I have replaced  Brownian index with "variation index" along $\pi$. I think the word "Brownian index" is confusing.}
At $H=1/2$, the process has the covariance of a Brownian bridge and
$1/2$-H\"older continuous paths. Our first theorem shows that it also shares
the Brownian bridge's quadratic variation along every fixed refining
sequence.

\begin{theorem}[Quadratic variation along refining partitions]
  \label{thm:main refining}
For every fixed deterministic refining sequence $\pi$, there is an
event $\Omega_\pi$ with $\Pp(\Omega_\pi)=1$ such that, for every
$\omega\in\Omega_\pi$,
\[
  \lim_{m\to\infty}\sup_{t\in[0,1]}
  \left|[X(\omega)]_{(\pi_m)}(t)-t\right|=0.
\]
\end{theorem}

\begin{remark}
The event $\Omega_\pi$ may depend on the fixed sequence $\pi$. The theorem
does not assert that one probability-one event works for all partition
sequences simultaneously. This is the same order of quantifiers as in
L\'evy's theorem and permits the path-dependent constructions discussed in
the introduction.
\end{remark}

Without refinement, a sufficiently fast mesh decay still guarantees convergence.
The next result gives the analogue of Dudley's sufficient condition and
shows that its order is sharp as Fern\'andez de la Vega~\cite{DeLaVega1974} showed it to be for Brownian motion. In the sharpness assertion, failure occurs
for every sign array.

\begin{theorem}
\label{thm:sharp-logarithmic-mesh}
Let \(X=X^{1/2}\) be the random signed Takagi--Landsberg bridge.
\begin{enumerate}[{\rm (a)}]
\item
Let \((\pi_n)\) be a deterministic sequence of partitions, not necessarily
refining. If 
\[
  \mesh(\pi_n)=o\Big(\frac1{\log n}\Big),
\]
then
\[
  \sup_{t\in[0,1]}\left|[X]_{(\pi_n)}(t)-t\right|
  \longrightarrow0\qquad\text{almost surely}.
\]

\item
Set
\[
  \gamma:=\frac{4+\sqrt2}{24}
          =\frac{2+1/\sqrt2}{12}.
\]
There exists a fixed deterministic, nonrefining sequence
\((\widehat\pi_n)\) such that
\[
  \mesh(\widehat\pi_n)
  =O\Big(\frac1{\log n}\Big)
\]
and, for every sign array,
\[
  \liminf_{n\to\infty}[X]_{(\widehat\pi_n)}=1-\gamma,
  \qquad
  \limsup_{n\to\infty}[X]_{(\widehat\pi_n)}=1+\gamma.
\]
In particular, the little-\(o\) condition in the first assertion
cannot be replaced by a big-\(O\) condition.
\end{enumerate}
\end{theorem}

The symmetry assumption also matters. The next theorem shows that a
common nonzero mean of the signs can change the quadratic variation along
refining sequences. Uniform thirds grids yield a limit larger than the
dyadic one, whereas triadic grids yield two distinct subsequential limits.

\purba[inline]{For the theorem below it's probably better to use the notation Y instade of X. I am making these changes below}
\begin{theorem}
  \label{thm:biased-signs}
Fix any $\nu\in(0,1)\setminus\{1/2\}$, and let
$\{\theta_{m,k}^{(\nu)}\}$ be independent random variables with
\[
\Pp\bigl(\theta_{m,k}^{(\nu)}=1\bigr)=\nu,
  \qquad
  \Pp\bigl(\theta_{m,k}^{(\nu)}=-1\bigr)=1-\nu.
\]
Define
\[
  Y_t^{(\nu)} :=\sum_{m=0}^{\infty}\sum_{k=0}^{2^m-1} \theta_{m,k}^{(\nu)}e_{m,k}(t),
  \qquad 0\le t\le1,
\]
and put $b:=2\nu-1$.  Consider the uniform thirds grids and the
triadic grids
\[
  \mathbb U_n
  :=\left\{\frac{j}{3\cdot2^n}:0\le j\le3\cdot2^n\right\},
  \qquad
  \mathbb T_n^{(3)}
  :=\left\{\frac{j}{3^n}:0\le j\le3^n\right\}.
\]
Then
\begin{equation}
  [Y^{(\nu)}]_{(\mathbb U_n)}
  \longrightarrow
  1+\frac{6+8\sqrt2}{9}\,b^2
  \qquad\text{almost surely}.
  \label{eq:biased-thirds-limit}
\end{equation}
In particular, this limit is strictly larger than the dyadic limit of $1$.

Along the triadic grids, the quadratic sums do not converge almost surely.
More precisely, there are deterministic subsequences $(n_j)$ and $(m_j)$
such that, almost surely,
\begin{align}
  [Y^{(\nu)}]_{(\mathbb T_{n_j}^{(3)})}      
  &\longrightarrow 1-\frac49b^2,
  \label{eq:biased-triadic-limit-one}\\
  [Y^{(\nu)}]_{(\mathbb T_{m_j}^{(3)})}
  &\longrightarrow 1-\frac{28}{81}b^2.
  \label{eq:biased-triadic-limit-two}
\end{align}
The two limits differ by $8b^2/81>0$.  By contrast, for every realization
of the signs,
\begin{equation}
  [Y^{(\nu)}]_{(\mathbb T_n)}=1-2^{-n},
  \label{eq:biased-dyadic-exact}
\end{equation}
and hence the dyadic quadratic sums converge to $1$.
\end{theorem}

\subsection{Critical variation at indices \texorpdfstring{$H\ne1/2$}{H different from 1/2}}
\purba[inline]{Have not checked/ modified this section so far}
For $H\ne1/2$, the critical exponent $p=1/H$ differs from two.  The next theorem
shows that the corresponding variation is not invariant under the choice
of a fixed deterministic refining sequence.

\begin{theorem}
  \label{thm:all-H-failure}
Let $H\in(0,1)\setminus\{1/2\}$ and $p=1/H$.  There is a fixed
deterministic refining sequence $\pi^H=(\pi_n^H)$ with vanishing mesh such
that
\[
  [X^H]_{(\pi_n^H)}^{(p)}
\]
fails to converge almost surely.  Hence $X^H$ does not admit continuous
critical $p$-th variation along $\pi^H$.

More precisely, for some integer $r=r(H)\ge1$, let
\begin{equation}
  \mathbb V_n^{(r)}
  :=\mathbb T_n\cup
    \left\{2^{-n}(k+2^{-r}):0\le k<2^n\right\}.
  \label{eq:split-grid-main}
\end{equation}
There is a deterministic constant $C_r(H)$ such that, almost surely,
\[
  [X^H]_{(\mathbb V_n^{(r)})}^{(p)}\longrightarrow C_r(H),
\]
and
\[
  C_r(H)>C_{\mathrm{dyad}}(H)\quad\text{if }H<\frac12,
  \qquad
  C_r(H)<C_{\mathrm{dyad}}(H)\quad\text{if }H>\frac12.
\]
With such an $r$ fixed, one may take, for $j\ge1$,
\begin{equation}
  \pi^H_{2j-1}:=\mathbb V_{jr}^{(r)},
  \qquad
  \pi^H_{2j}:=\mathbb T_{(j+1)r}.
  \label{eq:alternating-partition-main}
\end{equation}
The odd terminal sums then converge almost surely to $C_r(H)$ and the even
terminal sums to $C_{\mathrm{dyad}}(H)$.
\end{theorem}

Thus $H=1/2$ is the only index in this family at which critical variation
is stable along every fixed deterministic refining sequence.  The theorem
uses a split at an $H$-dependent dyadic fraction.  At $H=1/4$, the following
example identifies a competing limit explicitly on the simpler uniform
thirds grids.

\begin{proposition}[Explicit fourth-variation example]
  \label{prop:thirds-fourth}
Let $H=1/4$, and let $(\mathbb U_n)$ be the uniform thirds grids defined
in Theorem~\ref{thm:biased-signs}. Then
\begin{equation}
  [X^{1/4}]_{(\mathbb U_n)}^{(4)}
  \longrightarrow
  C_{\mathrm{thirds}}
  :=\frac{1301+1020\sqrt2}{1323}
  \qquad\text{almost surely}.
  \label{eq:thirds-limit}
\end{equation}
This differs from the dyadic value:
\begin{equation}
  C_{\mathrm{dyad}}(1/4)
  =\frac{13+12\sqrt2}{49},
  \qquad
  C_{\mathrm{thirds}}-C_{\mathrm{dyad}}(1/4)
  =\frac{950+696\sqrt2}{1323}>0.
  \label{eq:different-constants}
\end{equation}
\end{proposition}

\begin{remark}
Proposition~\ref{prop:thirds-fourth} makes the partition dependence concrete
at $H=1/4$: the dyadic and uniform thirds grids yield different
fourth-variation constants.  We do not claim that the thirds-grid constant
differs from $C_{\mathrm{dyad}}(H)$ for every $H\ne1/2$;
Theorem~\ref{thm:all-H-failure} establishes the general result through the
$H$-dependent split grids.
\end{remark}

\subsection{Invariance of variation index across partition}\label{sec:positive-critical}

Theorem~\ref{thm:all-H-failure} shows that the value of the critical
variation depends on the partition sequence when $H\ne1/2$. The results
of this subsection show that the critical generalized $p$-th variations are uniformly bounded
above over all partitions, for every sign array. For each fixed
deterministic sequence satisfying \eqref{eq:mesh-condition-critical},
their fluctuations around their expectations vanish almost surely,
their lower limit is positive almost surely, and the variation index is
$1/H$. Under this mesh condition, convergence of the terminal critical
sums is therefore entirely a first-moment property. The probability-one
event may depend on the prescribed sequence; no balance assumption is
needed.

Put 
\begin{equation}
  \kappa_H:=\min\{1,\,2-2H\}\in(0,1],
  \qquad H\in(0,1).
  \label{eq:kappa-H}
\end{equation}

\begin{proposition}[Deterministic upper bound]
\label{prop:deterministic-upper}
Let $H\in(0,1)$, $p=1/H$, and let $L_H$ be the constant of
Lemma~\ref{lem:uniform-holder}. For every finite partition $\tau$ of
$[0,1]$ and every sign array,
\begin{equation}
  [X^H]^{(p)}_{(\tau)}\le L_H^{\,p}.
  \label{eq:deterministic-upper}
\end{equation}
Consequently, for every partition sequence $\pi$ with vanishing mesh and
every sign array, $[X^H]^{(q)}_{(\pi_n)}\to0$ for all $q>p$, and
$p^\pi(X^H)\le1/H$.
\end{proposition}

\begin{theorem}[Asymptotic determinism of the critical sums]
\label{thm:critical-concentration}
Let $H\in(0,1)\setminus\{1/2\}$, $p=1/H$, and let $\pi$
be a fixed deterministic partition sequence with
\begin{equation}
  \mesh(\pi_n)=o\bigl((\log n)^{-1/\kappa_H}\bigr).
  \label{eq:mesh-condition-critical}
\end{equation}
Then there is a constant $c_H>0$, depending only on $H$, such that for all
sufficiently large $n$,
\begin{equation}
  c_H\le\E\,[X^H]^{(p)}_{(\pi_n)}\le L_H^{\,p},
  \label{eq:critical-mean-bounds}
\end{equation}
and
\begin{equation}
  [X^H]^{(p)}_{(\pi_n)}-\E\,[X^H]^{(p)}_{(\pi_n)}\longrightarrow0
  \qquad\text{almost surely}.
  \label{eq:critical-asymptotic-determinism}
\end{equation}
In particular, almost surely,
\begin{equation}
  c_H\le\liminf_{n\to\infty}[X^H]^{(p)}_{(\pi_n)}
  \le\limsup_{n\to\infty}[X^H]^{(p)}_{(\pi_n)}\le L_H^{\,p},
  \label{eq:critical-liminf-limsup}
\end{equation}
and $p^\pi(X^H)=1/H$.
\end{theorem}

For $H<1/2$ the mesh condition \eqref{eq:mesh-condition-critical} is
Dudley's condition $\mesh(\pi_n)=o(1/\log n)$ of
Theorem~\ref{thm:sharp-logarithmic-mesh}(a). For $H>1/2$, the
bounded-differences estimate below decays as $\mesh(\pi_n)^{2-2H}$,
which gives the stronger sufficient condition
$\mesh(\pi_n)=o((\log n)^{-1/(2-2H)})$.

\begin{corollary}[Convergence is a first-moment property]
\label{cor:first-moment}
Under the assumptions of Theorem~\ref{thm:critical-concentration}, the
following are equivalent:
\begin{enumerate}[{\rm (i)}]
\item $[X^H]^{(p)}_{(\pi_n)}$ converges almost surely;
\item $[X^H]^{(p)}_{(\pi_n)}$ converges in probability;
\item the deterministic sequence $\E\,[X^H]^{(p)}_{(\pi_n)}$ converges.
\end{enumerate}
When they hold, the almost-sure limit is $\lim_n\E[X^H]^{(p)}_{(\pi_n)}$.
\end{corollary}

Thus the limits in Theorem~\ref{thm:all-H-failure} and
Proposition~\ref{prop:thirds-fourth} can be identified through first
moments. For the sequence \eqref{eq:alternating-partition-main},
\[
  \mesh(\pi^H_{2j-1})=(1-2^{-r})2^{-jr},
  \qquad
  \mesh(\pi^H_{2j})=2^{-(j+1)r},
\]
so \eqref{eq:mesh-condition-critical} holds. Its terminal sums fail to
converge because their expectations converge along the odd and even
subsequences to the distinct constants $C_r(H)$ and
$C_{\mathrm{dyad}}(H)$, respectively. Similarly, the thirds grids satisfy
the mesh condition. Conversely, for each fixed deterministic sequence
satisfying \eqref{eq:mesh-condition-critical},
\eqref{eq:critical-liminf-limsup} places all subsequential limits in
$[c_H,L_H^p]$ almost surely.

The reciprocal $H$ of the variation index $1/H$ is the roughness exponent
in the convention of Han and Schied~\cite{HanSchied2025}. It is also the
Besov smoothness exponent $1/p$ appearing in
\cite[Lemma~3.2]{Cont2026}. That lemma assumes intrinsic $p$-roughness,
a property not established by the present concentration theorem. Here
the conclusion is that the variation index equals $1/H$ almost surely
along each fixed deterministic sequence satisfying
\eqref{eq:mesh-condition-critical}.

\section{Proofs}
\label{sec:proofs}

We now prove the results stated in Section~\ref{sec:main-results}.

\subsection{Probabilistic and regularity estimates}

We first record the moment inequalities used in the refining-partition
argument and a deterministic regularity estimate.

Let $\theta_1,\theta_2,\ldots$ be independent Rademacher variables.  A \emph{finite homogeneous Rademacher chaos of order two} is a
random variable of the form
\begin{equation}
  F=\sum_{1\le i<j\le N}a_{ij}\theta_i\theta_j,
  \qquad a_{ij}\in\mathbb R.
  \label{eq:order-two-chaos}
\end{equation}
Thus every monomial contains exactly two distinct Rademacher variables; in
particular, there are no constant, linear, or diagonal terms.  We use the
same terminology for an $L^2$ limit of finite chaoses of this form,
equivalently for an expansion
\[
  F=\sum_{i<j}a_{ij}\theta_i\theta_j,
  \qquad \sum_{i<j}a_{ij}^2<\infty,
\]
where the series converges in $L^2$.

\begin{lemma}[Rademacher hypercontractivity]\label{lem:bonami}
Let $F$ be a homogeneous Rademacher chaos of order two. Then
\[
  \lVert F\rVert_{L^4}\le 3\lVert F\rVert_{L^2}.
\]
\end{lemma}

For a finite chaos, this is the $q=4$, $d=2$ case of the
Bonami--Beckner hypercontractive inequality.  O'Donnell
\cite[Theorem~9.21]{ODonnell} proves the more general estimate
\[
  \lVert F\rVert_{L^q}
  \le(q-1)^{d/2}\lVert F\rVert_{L^2},
  \qquad q\ge2,
\]
for every Rademacher polynomial of degree at most $d$.  O'Donnell
\cite[Section~9.1]{ODonnell} also gives a direct proof of the $q=4$ case
under the name \emph{Bonami Lemma}.

For an $L^2$ limit of finite chaoses, we apply the finite inequality to the
difference of two finite truncations.  The truncations are then Cauchy in
$L^4$, and passage to the limit yields the stated inequality.

\purba[inline]{ Using \cite[Theorem 1.1]{RudelsonVershynin2013} and the Hanson–Wright tail bound one can prove the following lemma.
\begin{lemma}[Subgaussian hypercontractivity]\label{Lem: sub-gussian hypercontractivity} Let $(\xi_i)_{i\geq 1}$ be independent centered subgaussian random variables with $E[\xi_i^2]=1$ and the subgaussian Orlicz norm is bounded above i.e. $\|\xi_i\|_{\psi_2}\le K$. Then there exists a constant $C_K<\infty$, depending only on $K$, such that every finite degree--2 polynomial 
\[ F = \sum_i a_{ii}(\xi_i^2-1) + \sum_{i<j} a_{ij}\xi_i\xi_j \] satisfies \[ \|F\|_{L^4} \le C_K\, \|F\|_{L^2}. \]  \end{lemma}

}

A uniform subgaussian bound alone does not give the $L^4$--$L^2$ estimate
in Lemma~\ref{lem:bonami} for centered quadratic polynomials with diagonal
terms. For example, let $J_q$ be Bernoulli with parameter $q$, let $R$ be
an independent Rademacher variable, and set
$\zeta=R\sqrt{1+J_q-q}$. Then $\zeta$ is centered, has variance one, and is
bounded by $\sqrt2$, whereas
\[
  \frac{\|\zeta^2-1\|_{L^4}}{\|\zeta^2-1\|_{L^2}}
  \sim q^{-1/4}\qquad(q\downarrow0).
\]
For the extension to general subgaussian coefficients, we instead use the
operator moment bound \eqref{eq:subgaussian-operator-moment} below.

\begin{lemma}[Maximal moment inequality]\label{lem:moricz}
Let $D_1,D_2,\ldots$ be arbitrary random variables and let
$w_1,w_2,\ldots\ge0$.  Suppose that, for some $C<\infty$ and every
$p\le q$,
\[
  \E\left|\sum_{j=p}^{q}D_j\right|^4
  \le C\left(\sum_{j=p}^{q}w_j\right)^2.
\]
Then a constant $C'$ depending only on $C$ exists such that
\[
  \E\max_{p\le r\le q}
       \left|\sum_{j=p}^{r}D_j\right|^4
  \le C'\left(\sum_{j=p}^{q}w_j\right)^2.
\]
\end{lemma}

M\'oricz, Serfling, and Stout
\cite[Theorem~3.1]{MoriczSerflingStout} prove the required maximal
power-moment inequality for arbitrary random variables.  To obtain the
displayed form, we take their moment order to be $4$, their power exponent
to be $2$, and
\[
  g(p,q)=\sqrt C\sum_{j=p}^{q}w_j.
\]
This control is additive and hence superadditive, so the cited theorem
requires neither independence nor a martingale structure in this
application.  Its power exponent $2$ is strictly greater than one, as
required.

We also record a uniform deterministic regularity estimate \cite{MishuraSchied}.

\begin{lemma}[Uniform H\"older estimate]
  \label{lem:uniform-holder}
For every $H\in(0,1)$ there is a finite constant $L_H$, independent of the
sign array, such that the function $X^H$ defined by \eqref{eq:XH} satisfies
\begin{equation}
  |X^H_t-X^H_s|\le L_H|t-s|^H,
  \qquad s,t\in[0,1].
  \label{eq:H-holder}
\end{equation}
\end{lemma}

\begin{proof}
For $s=t$, both sides of \eqref{eq:H-holder} vanish.  We may therefore
assume $s\ne t$, set $\delta=|t-s|>0$, and choose the integer $N\ge0$ so that
$2^{-(N+1)}<\delta\le2^{-N}$.  The level-$m$ term of \eqref{eq:XH} has
Lipschitz constant at most $2^{m(1-H)}$ and sup norm at most
$2^{-mH-1}$.  Consequently,
\[
  |X^H_t-X^H_s|
  \le \delta\sum_{m<N}2^{m(1-H)}
      +\sum_{m\ge N}2^{-mH}.
\]
Each geometric sum is bounded by a constant depending only on $H$ times
$\delta^H$.  Hence \eqref{eq:H-holder} follows.
\end{proof}

\subsection{Quadratic variation at index \texorpdfstring{$H=1/2$}{H=1/2}}

We express every quadratic sum in the original dyadic Haar basis. This
keeps the independent Rademacher coefficients fixed as the partition
changes. A change to a partition-dependent Schauder basis would generally
destroy their independence: for independent Rademacher variables
$\theta_1,\theta_2$, the orthogonal transform
\[
  \eta_1=\frac{\theta_1+\theta_2}{\sqrt2},
  \qquad \eta_2=\frac{\theta_1-\theta_2}{\sqrt2}
\]
has uncorrelated coordinates of variance one, but
$\eta_1\eta_2=0$ almost surely. Coefficient transformations between
partition-dependent Schauder systems are studied in
\cite[Section~6.1]{ContDas}.

For a finite partition $\rho$ of $[0,t]$, with $0<t\le1$, let $I\in\rho$
denote its adjacent half-open intervals. We extend the terminal-sum
notation to this interval and set
\begin{equation}
  [X]_{(\rho)}:=\sum_{(u,v]\in\rho}(X_v-X_u)^2,
  \qquad s(\rho):=\sum_{(u,v]\in\rho}(v-u)^2,
  \qquad Z_\rho:=[X]_{(\rho)}-t+s(\rho).
  \label{eq:q-s-z-definitions}
\end{equation}
\purba[inline]{Maybe better to say $s(\rho):=\sum_{[u,v]\in\rho}|v-u|^2$  so that we completely avoid $|I|$ notation.}

\begin{lemma}[Quadratic-sum energy bounds]\label{lem:haar-energy}
For every finite partition $\rho$ of $[0,t]$, $Z_\rho$ is a homogeneous
Rademacher chaos of order two, and
\begin{equation}
  \E\bigl([X]_{(\rho)}\bigr)=t-s(\rho),\qquad
  \E|Z_\rho|^2\le2s(\rho),\qquad
  \E|Z_\rho|^4\le324s(\rho)^2.
  \label{eq:mean-q}
\end{equation}
If $\rho$ and $\sigma$ are partitions of the same interval $[0,t]$, and
$\tau=\rho\cup\sigma$ is their common refinement, then
\begin{align}
  \E|Z_\rho-Z_\sigma|^2
    &\le2\bigl(s(\rho)+s(\sigma)-2s(\tau)\bigr),
    \label{eq:L2-block-gen}\\
  \E|Z_\rho-Z_\sigma|^4
    &\le324\bigl(s(\rho)+s(\sigma)-2s(\tau)\bigr)^2.
    \label{eq:L4-block-gen}
\end{align}
In particular, if $\sigma$ refines $\rho$, then
\begin{align}
  \E|Z_\rho-Z_\sigma|^2&\le2\bigl(s(\rho)-s(\sigma)\bigr),
    \label{eq:L2-block}\\
  \E|Z_\rho-Z_\sigma|^4&\le324\bigl(s(\rho)-s(\sigma)\bigr)^2.
    \label{eq:L4-block}
\end{align}
\end{lemma}

\begin{proof}
Let $h_{n,k}=e'_{n,k}$ almost everywhere. These Haar functions form an
orthonormal basis of
\[
  H_0:=\left\{f\in L^2[0,1]:\int_0^1f(u)\,du=0\right\};
\]
adjoining the constant function $\one_{[0,1]}$ gives an orthonormal basis
of $L^2[0,1]$. Let $P$ be the orthogonal projection onto $H_0$, so that
$Pf=f-(\int_0^1f)\one_{[0,1]}$. For $I=(u,v]$,
\begin{equation}
  X_v-X_u=\sum_{n,k}\theta_{n,k}\langle\one_I,h_{n,k}\rangle
  \quad\text{in }L^2(\Omega).
  \label{eq:increment-haar}
\end{equation}
All inner products below are in $L^2[0,1]$.

Define the positive finite-rank operators
\[
  K_\rho:=\sum_{I\in\rho}\one_I\otimes\one_I,
  \qquad A_\rho:=PK_\rho P
     =\sum_{I\in\rho}P\one_I\otimes P\one_I,
\]
where $(f\otimes g)v=\langle v,g\rangle f$. Their trace and
Hilbert--Schmidt norm satisfy
\begin{align}
  \tr(A_\rho)&=\sum_{I\in\rho}(|I|-|I|^2)=t-s(\rho),
     \label{eq:trace}\\
  \|K_\rho\|_{\HS}^2&=\sum_{I\in\rho}|I|^2=s(\rho).
     \label{eq:Knorm}
\end{align}
Here $|I|=v-u$ for $I=(u,v]$. The second identity follows because the
nonzero eigenvalues of $K_\rho$ are the interval lengths $|I|$.

To justify the quadratic expansion, set
\[
  \mathcal I:=\{(n,k):n\ge0,\ 0\le k<2^n\},\qquad
  \mathcal I_N:=\{(n,k)\in\mathcal I:n<N\},
\]
and order $\mathcal I$ lexicographically. We write $\theta_\alpha$, $e_\alpha$, and
$h_\alpha$ for the corresponding signs, tents, and Haar functions, and $P_N$ for
the orthogonal projection onto their span with $\alpha\in\mathcal I_N$.
The truncated series
\[
  X^{[N]}:=\sum_{\alpha\in\mathcal I_N}\theta_\alpha e_\alpha
\]
is the linear interpolation of $X$ on $\mathbb T_N$ and converges
uniformly to $X$. With
$(A_\rho)_{\alpha\beta}:=\langle h_\alpha,A_\rho h_\beta\rangle$, its
quadratic sum is
\[
  [X^{[N]}]_{(\rho)}
  =\sum_{\alpha,\beta\in\mathcal I_N}
       (A_\rho)_{\alpha\beta}\theta_\alpha\theta_\beta.
\]
Since $\theta_\alpha^2=1$, the diagonal terms tend to $\tr(A_\rho)$.
The products $\theta_\alpha\theta_\beta$, $\alpha<\beta$, are orthonormal
in $L^2(\Omega)$, and $A_\rho$ is Hilbert--Schmidt. Thus the off-diagonal
terms converge in $L^2$, giving
\begin{equation}
  [X]_{(\rho)}=\tr(A_\rho)
    +2\sum_{\alpha<\beta}(A_\rho)_{\alpha\beta}
           \theta_\alpha\theta_\beta.
  \label{eq:quadratic-form}
\end{equation}
In particular,
\begin{equation}
  Z_\rho=2\sum_{\alpha<\beta}(A_\rho)_{\alpha\beta}
                 \theta_\alpha\theta_\beta
  \quad\text{in }L^2(\Omega).
  \label{eq:centered-chaos}
\end{equation}

For another partition $\sigma$ of $[0,t]$, the nonempty intersections of
intervals of $\rho$ and $\sigma$ are the intervals of $\tau$. Hence
\begin{equation}
  \tr(K_\rho K_\sigma)
    =\sum_{I\in\rho,J\in\sigma}|I\cap J|^2=s(\tau),
  \label{eq:cross-trace}
\end{equation}
and therefore
\begin{equation}
  \|K_\rho-K_\sigma\|_{\HS}^2
    =s(\rho)+s(\sigma)-2s(\tau).
  \label{eq:exact-energy-decrement}
\end{equation}
Compression by $P$ cannot increase the Hilbert--Schmidt norm. Applying
orthogonality to \eqref{eq:centered-chaos}, we obtain
\begin{align}
  \E|Z_\rho-Z_\sigma|^2
  &=4\sum_{\alpha<\beta}
       |(A_\rho-A_\sigma)_{\alpha\beta}|^2\notag\\
  &\le2\|A_\rho-A_\sigma\|_{\HS}^2
   \le2\bigl(s(\rho)+s(\sigma)-2s(\tau)\bigr).
  \label{eq: Z-increment}
\end{align}
The same calculation with $A_\rho$ alone gives
$\E|Z_\rho|^2\le2s(\rho)$. Lemma~\ref{lem:bonami} gives all the
fourth-moment bounds, since $3^4\cdot2^2=324$. If $\sigma$ refines $\rho$,
then $\tau=\sigma$, yielding the final assertions.
\end{proof}

\begin{remark}[Subgaussian coefficients]
The mean and moment bounds above extend, with constants depending only on $K$, to
\begin{equation}
  \widetilde X_t:=\sum_{n,k}\zeta_{n,k}e_{n,k}(t),
  \qquad 0\le t\le1,
  \label{eq:Y}
\end{equation}
where the coefficients are independent, centered, have variance one, and
satisfy $\sup_{n,k}\|\zeta_{n,k}\|_{\psi_2}\le K<\infty$. Here
$\|\zeta\|_{\psi_2}:=\inf\{a>0:\E\exp(\zeta^2/a^2)\le2\}$.
The subgaussian tail bound and a union bound give
$\max_{k<2^n}|\zeta_{n,k}|=O(\sqrt{n+1})$ almost surely, so the series
converges uniformly and defines a continuous bridge.

For a finite coefficient vector $\zeta$ and a symmetric matrix $B$,
integration of the Hanson--Wright inequality~\cite[Theorem~1.1]{RudelsonVershynin2013} gives
\begin{equation}
  \left\|\zeta^{\mathsf T}B\zeta-\tr B\right\|_{L^4}
  \le C_K\|B\|_{\HS}.
  \label{eq:subgaussian-operator-moment}
\end{equation}
Indeed, the tail is bounded by
$2\exp[-c_K\min(u^2/\|B\|_{\HS}^2,u/\|B\|_{\mathrm{op}})]$, and
$\|B\|_{\mathrm{op}}\le\|B\|_{\HS}$. Applying
\eqref{eq:subgaussian-operator-moment} to differences of finite Haar
compressions permits passage to the limit in $L^4$. Thus
\begin{equation}
  \widetilde Z_\rho
  :=[\widetilde X]_{(\rho)}-t+s(\rho)
  =2\sum_{\alpha<\beta}(A_\rho)_{\alpha\beta}\zeta_\alpha\zeta_\beta
     +\sum_\alpha(A_\rho)_{\alpha\alpha}(\zeta_\alpha^2-1).
  \label{eq:centered-chaos-gen}
\end{equation}
The diagonal term is generally random. In particular, this is not a
Rademacher chaos, and the constants in Lemma~\ref{lem:haar-energy} need
not remain $2$ and $324$. Nevertheless,
\[
  \|\widetilde Z_\rho-\widetilde Z_\sigma\|_{L^4}
  \le C_K\|A_\rho-A_\sigma\|_{\HS}
  \le C_K\bigl(s(\rho)+s(\sigma)-2s(\rho\cup\sigma)\bigr)^{1/2}.
\]
The same bound with a single partition proves
$\|\widetilde Z_\rho\|_{L^4}\le C_Ks(\rho)^{1/2}$. Consequently the proof
of Theorem~\ref{thm:main refining} below applies to $\widetilde X$ as well.
This includes standard normal coefficients and the bounded variance-one
laws $\operatorname{Uniform}(-\sqrt3,\sqrt3)$,
$\sqrt{20}(\operatorname{Beta}(2,2)-1/2)$, and
$\sqrt8(\operatorname{Beta}(1/2,1/2)-1/2)$.
Related Schauder constructions with non-Gaussian coefficients appear in
\cite{BayraktarDasKim2025}.
\end{remark}

\begin{proof}[Proof of Theorem~\ref{thm:main refining}]
Fix $t\in(0,1]$ and restrict the partitions to $[0,t]$ by setting
\[
  \rho_m(t):=(\pi_m\cap[0,t])\cup\{t\},\qquad
  s_m:=s(\rho_m(t)),\qquad Z_m:=Z_{\rho_m(t)}.
\]
Then $[X]_{(\rho_m(t))}=[X]_{(\pi_m)}(t)$, and $(\rho_m(t))$ is refining.
Consequently, $s_m$ is nonincreasing and
\begin{equation}
  0\le s_m\le t\,\mesh(\pi_m)\longrightarrow0.
  \label{eq:s-to-zero}
\end{equation}
Set $D_j=Z_j-Z_{j+1}$ and $w_j=s_j-s_{j+1}$. By
\eqref{eq:L4-block}, for $a\le b$,
\[
  \E\left|\sum_{j=a}^{b}D_j\right|^4
  \le324\left(\sum_{j=a}^{b}w_j\right)^2.
\]
Lemma~\ref{lem:moricz} and monotone convergence therefore give
\[
  \E\sup_{m\ge a}|Z_m-Z_a|^4\le C s_a^2.
\]
Together with \eqref{eq:mean-q} and $(u+v)^4\le8(u^4+v^4)$ for
$u,v\ge0$, this yields
\begin{equation}
  \E\sup_{m\ge a}|Z_m|^4\le C' s_a^2\longrightarrow0.
  \label{eq:max-tail}
\end{equation}
Fatou's lemma now implies $\limsup_m|Z_m|=0$ almost surely. Since
$[X]_{(\pi_m)}(t)=Z_m+t-s_m$, we obtain
\begin{equation}
  [X]_{(\pi_m)}(t)\longrightarrow t\qquad\text{almost surely}.
  \label{eq:fixed-t}
\end{equation}
At $t=0$, this holds by definition.

Intersect the probability-one events in \eqref{eq:fixed-t} over rational
$t\in[0,1]$. On this event, \eqref{eq:successor-comparison} implies that
the nondecreasing functions $V_{\pi_m}(X;t)$ converge to
$t$ at every rational time. For each integer $k\ge1$, monotonicity between
the points $j/k$, $0\le j\le k$, gives
\[
    \sup_{t\in[0,1]}
  \left|V_{\pi_m}(X;t)-t\right|
  \le\max_{0\le j\le k}
     \left|V_{\pi_m}(X;j/k)-j/k\right|+\frac1k.
\]
First let $m\to\infty$ and then $k\to\infty$. The successor sums converge
uniformly to $t$, and \eqref{eq:successor-comparison} transfers this
convergence to the stopped sums. This proves the theorem.
\end{proof}

\begin{proof}[Proof of Theorem~\ref{thm:sharp-logarithmic-mesh}]
(a) We first derive a concentration bound for a finite partition $\rho$
of $[0,t]$, where $0<t\le1$. The normalized indicators
$|I|^{-1/2}\one_I$, $I\in\rho$, are eigenvectors of $K_\rho$ with
eigenvalues $|I|$. Together with orthogonal compression, this gives
\begin{equation}
  \lVert A_\rho\rVert_{\mathrm{op}}
  \le\lVert K_\rho\rVert_{\mathrm{op}}=\mesh(\rho),
  \qquad
  \lVert A_\rho\rVert_{\HS}^2
  \le s(\rho)\le t\mesh(\rho).
  \label{eq:A-op-HS-log}
\end{equation}
Let $A_{\rho,N}=P_NA_\rho P_N$ and retain the truncation $X^{[N]}$
from the proof of Lemma~\ref{lem:haar-energy}. Its quadratic sum is
the quadratic form associated with $A_{\rho,N}$ and the independent
Rademacher coefficients $(\theta_\alpha)_{\alpha\in\mathcal I_N}$.
The Hanson--Wright inequality
\cite[Theorem~1.1]{RudelsonVershynin2013} therefore yields a universal
constant $c>0$ such that, for every $u>0$,
\begin{equation}
  \Pp\left(
    \left|[X^{[N]}]_{(\rho)}-\tr(A_{\rho,N})\right|>u
  \right)
  \le2\exp\left(
    -c\min\left\{
      \frac{u^2}{s(\rho)},\frac{u}{\mesh(\rho)}
    \right\}\right).
  \label{eq:finite-HW-log}
\end{equation}
For each fixed $\rho$, uniform convergence of $X^{[N]}$ gives
$[X^{[N]}]_{(\rho)}\to[X]_{(\rho)}$ for every sign array, and
trace-norm convergence gives
$\tr(A_{\rho,N})\to\tr(A_\rho)=t-s(\rho)$. Fatou's lemma applied
to the indicators of the strict exceedance events in
\eqref{eq:finite-HW-log} now yields
\begin{equation}
  \Pp(|Z_\rho|>u)
  \le2\exp\left(
    -c\min\left\{
      \frac{u^2}{s(\rho)},\frac{u}{\mesh(\rho)}
    \right\}\right)
  \le2\exp\left(-\frac{c\min\{u^2,u\}}{\mesh(\rho)}\right).
  \label{eq:infinite-HW-log}
\end{equation}

Fix $t\in(0,1]$ and put
$\rho_n(t)=(\pi_n\cap[0,t])\cup\{t\}$. Since
$\mesh(\rho_n(t))\le\mesh(\pi_n)=o(1/\log n)$, the last bound is
summable in $n$ for every fixed $u>0$. Borel--Cantelli, applied with
$u=1/j$, gives $Z_{\rho_n(t)}\to0$ almost surely. Moreover,
$s(\rho_n(t))\le t\mesh(\pi_n)\to0$, and hence
\[
  [X]_{(\pi_n)}(t)
  =Z_{\rho_n(t)}+t-s(\rho_n(t))\longrightarrow t
  \qquad\text{almost surely}.
\]
We intersect these events over rational $t\in[0,1]$. The monotone
successor sums $V_{\pi_m}(X;\cdot)$ and \eqref{eq:successor-comparison} then give uniform
convergence on this event, by the argument at the end of the proof of
Theorem~\ref{thm:main refining}.

\medskip
(b) For $\ell\ge1$, let $\mathcal P_\ell$ contain every partition
obtained from $\mathbb T_{\ell-1}$ by choosing separately in each of
its $2^{\ell-1}$ intervals whether to insert the midpoint. Thus
\begin{equation}
  |\mathcal P_\ell|=2^{\,2^{\ell-1}},
  \qquad
  \mesh(\pi)\le2^{1-\ell}
  \quad(\pi\in\mathcal P_\ell).
  \label{eq:block-cardinality-mesh}
\end{equation}
List the blocks $\mathcal P_1,\mathcal P_2,\ldots$ in order, and
choose a fixed ordering within each block, beginning with
$\mathbb T_\ell$ and ending with $\mathbb T_{\ell-1}$. This defines
a deterministic nonrefining sequence $(\widehat\pi_n)$. If
$\widehat\pi_n$ occurs in block $\ell$, then
\[
  n\le\sum_{q=1}^{\ell}2^{\,2^{q-1}}
  <2^{\,1+2^{\ell-1}}.
\]
Consequently, for $n\ge2$,
\[
  \mesh(\widehat\pi_n)\log n
  <(1+2^{1-\ell})\log2\le2\log2,
\]
which proves the required mesh bound.

We determine the extrema within each block. Set $\lambda=2^{-1/2}$
and define the normalized dyadic increments
\[
  W_{r,k}:=2^{r/2}
    \bigl(X_{(k+1)2^{-r}}-X_{k2^{-r}}\bigr),
  \qquad 0\le k<2^r.
\]
The levels below $r$ are affine on the $k$th level-$r$ interval,
the level-$r$ tent has midpoint value $2^{-r/2-1}\theta_{r,k}$,
and all finer levels vanish at the endpoints and midpoint. Therefore
\begin{equation}
  W_{r+1,2k}=\lambda(W_{r,k}+\theta_{r,k}),
  \qquad
  W_{r+1,2k+1}=\lambda(W_{r,k}-\theta_{r,k}).
  \label{eq:dyadic-increment-recursion-log}
\end{equation}
Since $W_{0,0}=0$, induction shows that, for every function $f$,
\begin{equation}
  2^{-r}\sum_{k=0}^{2^r-1}f(W_{r,k})=\E f(W_r),
  \qquad
  W_r:=\sum_{j=1}^{r}\lambda^jY_j,
  \label{eq:exact-cube-average-log}
\end{equation}
where $(Y_j)$ is an auxiliary independent Rademacher sequence.
Indeed, each pair in \eqref{eq:dyadic-increment-recursion-log}
contains both signs, regardless of $\theta_{r,k}$. In particular,
\begin{equation}
  [X]_{(\mathbb T_\ell)}
  =\E W_\ell^2=\sum_{j=1}^{\ell}2^{-j}=1-2^{-\ell}.
  \label{eq:exact-dyadic-qv-log}
\end{equation}

Fix $\ell$ and set $r=\ell-1$ and $h=2^{-\ell}$. In the $k$th
level-$r$ interval, the contributions obtained by inserting and by
omitting the midpoint are, respectively,
\[
  h(W_{r,k}^2+1)
  \qquad\text{and}\qquad
  2hW_{r,k}^2.
\]
Thus omitting the midpoint changes the quadratic sum by
$h(W_{r,k}^2-1)$. Every combination of these choices occurs in
$\mathcal P_\ell$, so \eqref{eq:exact-cube-average-log} gives
\begin{align}
  M_\ell
  &:=\max_{\pi\in\mathcal P_\ell}[X]_{(\pi)}
   =1-2^{-\ell}
     +\frac12\E[(W_{\ell-1}^2-1)_+],
  \label{eq:block-maximum-log}\\
  m_\ell
  &:=\min_{\pi\in\mathcal P_\ell}[X]_{(\pi)}
   =1-2^{-\ell}
     -\frac12\E[(1-W_{\ell-1}^2)_+].
  \label{eq:block-minimum-log}
\end{align}
Both extrema are deterministic, although the partitions attaining
them may depend on the signs.

The variables $W_r$ converge uniformly in the auxiliary signs to
\[
  W:=\sum_{j=1}^{\infty}2^{-j/2}Y_j=\sqrt2\,U+V,
  \qquad
  U:=\sum_{j=1}^{\infty}2^{-j}Y_{2j-1},
  \quad
  V:=\sum_{j=1}^{\infty}2^{-j}Y_{2j}.
\]
The variables $U$ and $V$ are independent and uniform on $[-1,1]$.
Writing $b=\sqrt2$, the convolution density of $W$ is
$f_W(x)=(1+b-|x|)/(4b)$ for $1\le|x|\le1+b$. Hence
\begin{equation}
  \E[(W^2-1)_+]
  =\frac1{2b}\int_1^{1+b}(x^2-1)(1+b-x)\,dx
  =\frac{4+\sqrt2}{12}=2\gamma.
  \label{eq:positive-part-constant-log}
\end{equation}
Since $\E W^2=1$, we also have $\E[(1-W^2)_+]=2\gamma$.
Bounded convergence in
\eqref{eq:block-maximum-log}--\eqref{eq:block-minimum-log} yields
$M_\ell\to1+\gamma$ and $m_\ell\to1-\gamma$. Every quadratic sum
in block $\ell$ lies between these extrema, and both extrema are
attained. It follows that
\[
  \liminf_{n\to\infty}[X]_{(\widehat\pi_n)}=1-\gamma,
  \qquad
  \limsup_{n\to\infty}[X]_{(\widehat\pi_n)}=1+\gamma.
\]
The construction and all identities in part~(b) hold for every sign
array.
\end{proof}

\begin{proof}[Proof of Theorem~\ref{thm:biased-signs}]
The identity \eqref{eq:exact-dyadic-qv-log} holds for every sign array,
so it also proves \eqref{eq:biased-dyadic-exact}.

To separate the effect of the bias, define the periodic tent function and
the all-positive path by
\[
  \varphi(t):=\operatorname{dist}(t,\mathbb Z),
  \qquad
  F(t):=\sum_{m=0}^{\infty}2^{-m/2}\varphi(2^mt).
\]
Set $\zeta_{m,k}^{(\nu)}:=\theta_{m,k}^{(\nu)}-b$ and
$\sigma^2:=\E(\zeta_{m,k}^{(\nu)})^2=1-b^2$.  Then
\begin{equation}
  Y^{(\nu)}=bF+\widetilde Y,
  \qquad
  \widetilde Y_t
  :=\sum_{m=0}^{\infty}\sum_{k=0}^{2^m-1}\zeta_{m,k}^{(\nu)}e_{m,k}(t),
  \label{eq:biased-decomposition}
\end{equation}
where the coefficients of $\widetilde Y$ are independent, centered,
and bounded. We abbreviate these coefficients to $\zeta_\alpha^{(\nu)}$
when using the single Haar index $\alpha$.

For integers $q\ge1$, let
$\tau_q:=\{j/q:0\le j\le q\}$ and write
$\Delta_j f:=f((j+1)/q)-f(j/q)$ in calculations on this grid.
The operators from Lemma~\ref{lem:haar-energy} satisfy
\[
  \tr(A_{\tau_q})=1-q^{-1},
  \qquad
  \lVert A_{\tau_q}\rVert_{\HS}^2
  \le\lVert K_{\tau_q}\rVert_{\HS}^2=q^{-1}.
\]
Expanding the quadratic form in the Haar basis gives
\begin{equation}
  \E\bigl([\widetilde Y]_{(\tau_q)}\bigr)
  =\sigma^2(1-q^{-1}),
  \qquad
  \operatorname{Var}\bigl([\widetilde Y]_{(\tau_q)}\bigr)
  \le C_\nu q^{-1}.
  \label{eq:biased-centered-qv-bound}
\end{equation}
Indeed, the variance is
\[
  \operatorname{Var}((\zeta_{0,0}^{(\nu)})^2)
       \sum_\alpha(A_{\tau_q})_{\alpha\alpha}^2
  +2\sigma^4\sum_{\alpha\ne\beta}
       (A_{\tau_q})_{\alpha\beta}^2.
\]
These identities follow first for finite truncations; boundedness of the
coefficients and convergence of the compressed operators in trace and
Hilbert--Schmidt norms justify passage to the limit.

For the cross term, set
\[
  C_q:=\sum_{j=0}^{q-1}\Delta_j F\,\Delta_j\widetilde Y,
  \qquad
  g_q:=\sum_{j=0}^{q-1}(\Delta_j F)\one_{(j/q,(j+1)/q]}.
\]
Since $F(0)=F(1)=0$, we have $Pg_q=g_q$.  Parseval's identity and
Lemma~\ref{lem:uniform-holder} yield
\begin{equation}
  \E C_q=0,
  \qquad
  \operatorname{Var}(C_q)
  =\sigma^2\lVert g_q\rVert_2^2
  =\frac{\sigma^2}{q}[F]_{(\tau_q)}
  \le C_\nu' q^{-1}.
  \label{eq:biased-cross-bound}
\end{equation}
The bounds in \eqref{eq:biased-centered-qv-bound} and
\eqref{eq:biased-cross-bound} are summable along both $q=3\cdot2^n$
and $q=3^n$.  Chebyshev's inequality and the Borel--Cantelli lemma
therefore imply, on one event of probability one,
\begin{equation}
  [Y^{(\nu)}]_{(\tau_q)}
  =\sigma^2+b^2[F]_{(\tau_q)}+o(1)
  \label{eq:biased-reduction}
\end{equation}
along both sequences.  It remains to compute the deterministic sums.

For the thirds grids, put
\[
  a:=F(1/3)=F(2/3)
    =\frac13\sum_{m=0}^{\infty}2^{-m/2}
    =\frac{2+\sqrt2}{3}.
\]
On each level-$n$ dyadic cell, the levels below $n$ are affine and
the remaining levels form a copy of $2^{-n/2}F$.  If
$d_{n,k}:=F((k+1)2^{-n})-F(k2^{-n})$, the three increments within
that cell are consequently
\[
  \frac{d_{n,k}}3+2^{-n/2}a,
  \qquad \frac{d_{n,k}}3,
  \qquad \frac{d_{n,k}}3-2^{-n/2}a.
\]
Squaring and summing gives the exact identity
\begin{equation}
  [F]_{(\mathbb U_n)}
  =\frac13[F]_{(\mathbb T_n)}+2a^2
  =\frac13(1-2^{-n})+\frac{2(2+\sqrt2)^2}{9}.
  \label{eq:biased-thirds-deterministic}
\end{equation}
Its limit is $(15+8\sqrt2)/9$.  Since $\sigma^2=1-b^2$,
\eqref{eq:biased-reduction} proves \eqref{eq:biased-thirds-limit}.

To treat the triadic grids, we establish an asymptotic formula for every
odd $q$.  Write
\[
  L:=\lfloor\log_2q\rfloor,
  \qquad c_q:=q2^{-L}\in[1,2),
\]
and define
\begin{align}
  \mathcal R(u)
  &:=\int_0^1\bigl(\varphi(x+u)-\varphi(x)\bigr)^2\,dx\notag\\
  &=d(u)^2-\frac43d(u)^3,
  \qquad d(u):=\operatorname{dist}(u,\mathbb Z),
  \label{eq:biased-tent-energy}\\
  G(c)
  &:=c\sum_{\ell\in\mathbb Z}2^{-\ell}
        \mathcal R(2^\ell/c),
  \qquad 1\le c\le2.
  \label{eq:biased-profile}
\end{align}
The formula for $\mathcal R$ follows by integration on the linear pieces
of $\varphi$.  The summands defining $G$ are bounded uniformly in $c$
by constant multiples of $2^{-\ell}$ for $\ell\ge0$ and $2^\ell$
for $\ell<0$.  Thus the series converges uniformly and $G$ is continuous.
We claim that
\begin{equation}
  [F]_{(\tau_q)}-G(c_q)\longrightarrow0
  \qquad\text{as odd }q\longrightarrow\infty.
  \label{eq:biased-odd-grid-asymptotic}
\end{equation}

The first step is to compare the discrete sum with its spatial average.
Using the periodic extension of $F$, set
\[
  D_q(t):=\sqrt q\bigl(F(t+q^{-1})-F(t)\bigr).
\]
The contribution of level $m$ is bounded by
\[
  \min\left\{\frac{2^{m/2}}{\sqrt q},
                   \frac{\sqrt q}{2}\,2^{-m/2}\right\}.
\]
Consequently, $\lVert D_q\rVert_\infty$ is bounded uniformly in $q$.
For a fixed integer $R\ge0$ and $L\ge R$, let $D_q^{[R]}$ retain
only the levels $L-R,\ldots,L+R$.  Summing the omitted geometric tails
gives
$\lVert D_q-D_q^{[R]}\rVert_\infty=O(2^{-R/2})$, uniformly in $q$.
Define
\[
  H_{R,c}(x)
  :=\sum_{\ell=-R}^{R}\sqrt c\,2^{-\ell/2}
       \left[
         \varphi(2^{\ell+R}x+2^\ell/c)
         -\varphi(2^{\ell+R}x)
       \right].
\]
Then $D_q^{[R]}(t)=H_{R,c_q}(2^{L-R}t)$.
Because $q$ is odd, multiplication by $2^{L-R}$ permutes the residues
modulo $q$.  Periodicity therefore gives
\begin{align*}
  \frac1q\sum_{j=0}^{q-1}\bigl|D_q^{[R]}(j/q)\bigr|^2
  &=\frac1q\sum_{j=0}^{q-1}\bigl|H_{R,c_q}(j/q)\bigr|^2,\\
  \int_0^1\bigl|D_q^{[R]}(t)\bigr|^2\,dt
  &=\int_0^1\bigl|H_{R,c_q}(x)\bigr|^2\,dx.
\end{align*}
For fixed $R$, the functions $H_{R,c}^2$, $1\le c\le2$, have a common
Lipschitz bound, so the difference between the last Riemann sum and
integral is $O_R(q^{-1})$.  Restoring the omitted levels gives an error
of $O(2^{-R/2})+O_R(q^{-1})$.  Letting $q\to\infty$ through odd
integers and then $R\to\infty$ proves
\begin{equation}
  [F]_{(\tau_q)}
  -q\int_0^1\bigl(F(t+q^{-1})-F(t)\bigr)^2\,dt
  \longrightarrow0.
  \label{eq:biased-discrete-continuous}
\end{equation}

The level increments in this integral are orthogonal.  Indeed,
$\varphi(x+1/2)=1/2-\varphi(x)$, so translation by $2^{-m-1}$
changes the sign of the level-$m$ increment and leaves every finer-level
increment unchanged.  Their pairwise integrals therefore vanish, and
\begin{align*}
  q\int_0^1\bigl(F(t+q^{-1})-F(t)\bigr)^2\,dt
  &=q\sum_{m=0}^{\infty}2^{-m}\mathcal R(2^m/q)\\
  &=c_q\sum_{\ell=-L}^{\infty}2^{-\ell}
             \mathcal R(2^\ell/c_q).
\end{align*}
The omitted terms with $\ell<-L$ are $O(2^{-L})$, uniformly in $c_q$.
Together with \eqref{eq:biased-discrete-continuous}, this proves
\eqref{eq:biased-odd-grid-asymptotic}.

The negative-index terms in \eqref{eq:biased-profile} sum to
$c^{-1}-4/(9c^2)$.  For $c=1$, all remaining terms vanish; for
$c=3/2$, they have $d(2^\ell/c)=1/3$ and
$\mathcal R(2^\ell/c)=5/81$.  Hence
\begin{equation}
  G(1)=\frac59,
  \qquad
  G(3/2)=\frac23-\frac{16}{81}+\frac{15}{81}
        =\frac{53}{81}.
  \label{eq:biased-profile-values}
\end{equation}
Finally,
\[
  c_{3^n}=2^{\{n\log_2 3\}}.
\]
Since $\log_2 3$ is irrational, the fractional parts
$\{n\log_2 3\}$ are dense in $[0,1]$.  We may therefore choose
deterministic subsequences $(n_j)$ and $(m_j)$ such that
$c_{3^{n_j}}\to1$ and $c_{3^{m_j}}\to3/2$.
Continuity of $G$, \eqref{eq:biased-reduction}, and
\eqref{eq:biased-odd-grid-asymptotic} now give, almost surely,
\begin{align*}
  [Y^{(\nu)}]_{(\mathbb T_{n_j}^{(3)})}
  &\longrightarrow \sigma^2+\frac59b^2=1-\frac49b^2,\\
  [Y^{(\nu)}]_{(\mathbb T_{m_j}^{(3)})}
  &\longrightarrow \sigma^2+\frac{53}{81}b^2
                   =1-\frac{28}{81}b^2.
\end{align*}
These are \eqref{eq:biased-triadic-limit-one} and
\eqref{eq:biased-triadic-limit-two}, and they are distinct because $b\ne0$.
\end{proof}

\subsection{Critical variation}

We now prove Theorem~\ref{thm:all-H-failure} and
Proposition~\ref{prop:thirds-fourth}.  We first establish a common limit
formula for partitions obtained by repeating a fixed template in each
dyadic cell.  A small dyadic split then gives partition dependence for
every $H\ne1/2$, while the thirds template permits an explicit calculation
at $H=1/4$.

\subsubsection{Finite-template limits}

Fix $H\in(0,1)$, set $p=1/H$ and $\lambda=2^{H-1}$, and let
\[
  W_n:=\sum_{r=1}^{n}\lambda^rY_{r-1},
  \qquad
  W:=\lambda Z_H=\sum_{r=1}^{\infty}\lambda^rY_{r-1},
\]
where the auxiliary Rademacher sequence $(Y_m)$ from
\eqref{eq:known-dyadic-limit} is independent of $X^H$.
For a finite template $\boldsymbol u=\{0=u_0<u_1<\cdots<u_q=1\}$, define
\begin{equation}
  \mathbb P_n(\boldsymbol u)
  :=\left\{2^{-n}(k+u_j):
       0\le k<2^n,\ 0\le j\le q\right\},
  \label{eq:template-partition}
\end{equation}
with repeated endpoints included only once, and set
\begin{equation}
  G_{\boldsymbol u,p}(w,x)
  :=\sum_{j=0}^{q-1}
    \left|(u_{j+1}-u_j)w+x(u_{j+1})-x(u_j)\right|^p.
  \label{eq:template-functional}
\end{equation}
The sequence $\bigl(\mathbb P_n(\boldsymbol u)\bigr)$ need not be
refining; the following limit requires only that the template be fixed.

\begin{lemma}[Finite-template limits]
  \label{lem:finite-template}
For every sign array,
\begin{equation}
  [X^H]_{(\mathbb T_n)}^{(p)}
  =\E|W_n|^p
  \longrightarrow\E|W|^p=C_{\mathrm{dyad}}(H).
  \label{eq:dyadic-pvar}
\end{equation}
For every fixed finite template $\boldsymbol u$,
\begin{equation}
  [X^H]_{(\mathbb P_n(\boldsymbol u))}^{(p)}
  \longrightarrow
  C_{\boldsymbol u}(H)
  :=\E G_{\boldsymbol u,p}(W,X^H)
  \qquad\text{almost surely}.
  \label{eq:template-limit}
\end{equation}
\end{lemma}

\begin{proof}
On the dyadic cell $I_{n,k}=[k2^{-n},(k+1)2^{-n}]$, the levels $m<n$
are affine.  The level-$m$ term has slope of magnitude $2^{m(1-H)}$.
Writing $r=n-m$, its slope after rescaling the cell and dividing by
$2^{-nH}$ therefore has magnitude
$2^{nH-n}2^{(n-r)(1-H)}=\lambda^r$.
Consequently,
\begin{equation}
  X^H_{2^{-n}(k+u)}-X^H_{2^{-n}k}
  =2^{-nH}\bigl(uW_{n,k}+\widetilde X^{\,n,k}_u\bigr),
  \qquad 0\le u\le1,
  \label{eq:cell-decomposition}
\end{equation}
where, for suitable slope signs $\sigma^{(n)}_{r,k}\in\{-1,1\}$,
\begin{equation}
  W_{n,k}:=\sum_{r=1}^{n}\lambda^r\sigma^{(n)}_{r,k},
  \label{eq:coarse-slope}
\end{equation}
and the descendant terms form the process
\[
  \widetilde X^{\,n,k}_u
  :=\sum_{\ell=0}^{\infty}2^{\ell(1/2-H)}
      \sum_{v=0}^{2^\ell-1}
        \theta_{n+\ell,\,2^\ell k+v}e_{\ell,v}(u).
\]
Let $\mathcal F_{<n}$ be generated by the signs at levels below $n$.
The processes $\widetilde X^{\,n,k}$ use disjoint coefficient subtrees;
conditionally on $\mathcal F_{<n}$, they are independent copies of $X^H$.

For every sign array, the slope vectors
$(\sigma^{(n)}_{1,k},\ldots,\sigma^{(n)}_{n,k})$, $0\le k<2^n$,
enumerate $\{-1,1\}^n$ exactly once.  Indeed, bisecting a cell preserves
all its previous slope signs and adds the two opposite signs
$+\theta_{n,k}$ and $-\theta_{n,k}$.  Induction, with the coordinate order
reversed as in \eqref{eq:coarse-slope}, proves the assertion.
Since $\widetilde X^{\,n,k}_1=0$ and $Hp=1$, it follows that
\[
  [X^H]_{(\mathbb T_n)}^{(p)}
  =2^{-n}\sum_{k=0}^{2^n-1}|W_{n,k}|^p
  =\E|W_n|^p.
\]
The variables $W_n$ converge to $W$ and are uniformly bounded.
Thus bounded convergence and $\lambda^p=2^{1-p}$ give
\eqref{eq:dyadic-pvar}, with the normalization in
\eqref{eq:known-dyadic-limit}.

For the general template, \eqref{eq:cell-decomposition} gives
\begin{equation}
  [X^H]_{(\mathbb P_n(\boldsymbol u))}^{(p)}
  =2^{-n}\sum_{k=0}^{2^n-1}
       G_{\boldsymbol u,p}
       (W_{n,k},\widetilde X^{\,n,k}).
  \label{eq:template-average}
\end{equation}
Set $g(w):=\E G_{\boldsymbol u,p}(w,X^H)$.  Conditional independence
and the slope-vector enumeration imply
\begin{equation}
  \E\left[[X^H]_{(\mathbb P_n(\boldsymbol u))}^{(p)}
       \mid\mathcal F_{<n}\right]
  =2^{-n}\sum_{k=0}^{2^n-1}g(W_{n,k})
  =\E g(W_n)=:m_n.
  \label{eq:template-conditional-mean}
\end{equation}
Uniformly over the signs,
\[
  |W_{n,k}|\le\frac{\lambda}{1-\lambda},
  \qquad
  \lVert X^H\rVert_\infty\le\frac{1}{2(1-2^{-H})}.
\]
Hence $g$ is continuous and bounded on the common range of $W_n$ and
$W$, so $m_n\to\E g(W)$.  The same bounds give a deterministic constant
$B$ bounding every summand in \eqref{eq:template-average}.  Conditional
independence therefore yields
\[
  \E\left|[X^H]_{(\mathbb P_n(\boldsymbol u))}^{(p)}-m_n\right|^2
  \le B^2 2^{-n}.
\]
By Tonelli's theorem, the squared centered terms have a finite sum almost
surely and thus tend to zero.  Together with the limit of $m_n$, this
proves \eqref{eq:template-limit}.
\end{proof}

\subsubsection{The small-split comparison}

\begin{proof}[Proof of Theorem~\ref{thm:all-H-failure}]
Retain $p$, $\lambda$, and $W$ from Lemma~\ref{lem:finite-template}, and
write $M_p:=\E|W|^p=C_{\mathrm{dyad}}(H)$.  For $r\ge1$, let
$t_r=2^{-r}$ and $\boldsymbol u_r=\{0,t_r,1\}$, so that
$\mathbb P_n(\boldsymbol u_r)=\mathbb V_n^{(r)}$.
Lemma~\ref{lem:finite-template} gives
\begin{align}
  [X^H]_{(\mathbb V_n^{(r)})}^{(p)}
  &\longrightarrow C_r(H)
  \qquad\text{almost surely},\notag\\
  C_r(H)&:=\E\left[
       |t_rW+X^H_{t_r}|^p
       +|(1-t_r)W-X^H_{t_r}|^p
     \right].
  \label{eq:small-split-limit}
\end{align}

We compare $C_r(H)$ with $M_p$ as $r\to\infty$.  Write $t=t_r$ and
$a=t^H$.  At the dyadic point $t=2^{-r}$, all levels at least $r$
vanish and the remaining tents are on their increasing branches.  Thus
\begin{equation}
  X^H_t
  =t\sum_{m=0}^{r-1}2^{m(1-H)}\theta_{m,0}
  \stackrel d=aU_r,
  \qquad
  U_r:=\sum_{j=1}^{r}\lambda^j\xi_j,
  \label{eq:XH-small-dyadic-time}
\end{equation}
where $(\xi_j)$ is an independent Rademacher sequence, independent of
$W$.  Since $a\lambda^r=t$ and $U_r+\lambda^rW\stackrel d=W$, the
first term in \eqref{eq:small-split-limit} contributes $a^pM_p=tM_p$.
Symmetry of $U_r$ then gives
\begin{equation}
  C_r(H)=tM_p+\E|(1-t)W+aU_r|^p.
  \label{eq:Cr-exact}
\end{equation}
With $c_r=a/(1-t)$, we obtain the comparison identity
\begin{equation}
  C_r(H)-M_p
  =\bigl[t+(1-t)^p-1\bigr]M_p
   +(1-t)^p\left(\E|W+c_rU_r|^p-M_p\right).
  \label{eq:Cr-master}
\end{equation}
Here $W$ and all $U_r$ are bounded by $\lambda/(1-\lambda)$,
$c_r\to0$, and $\E U_r^2\to\E W^2>0$.

Suppose first that $H<1/2$, so $p>2$.  The function $f(x)=|x|^p$
has a continuous second derivative.  Taylor's formula, uniformly on the
common bounded range, and $\E U_r=0$ therefore yield
\begin{equation}
  \E|W+c_rU_r|^p
  =M_p+\frac{p(p-1)}2c_r^2
       \E|W|^{p-2}\,\E U_r^2+o(c_r^2).
  \label{eq:p-greater-two-expansion}
\end{equation}
Since $t+(1-t)^p-1=(1-p)t+O(t^2)$,
$(1-t)^pc_r^2=t^{2H}(1-t)^{p-2}$, and $t=o(t^{2H})$,
\eqref{eq:Cr-master} gives
\begin{equation}
  C_r(H)-M_p
  \sim\frac{p(p-1)}2\,t^{2H}
       \E|W|^{p-2}\,\E W^2>0.
  \label{eq:Cr-positive-asymptotic}
\end{equation}

Now suppose that $H>1/2$, so $1<p<2$.  We first note that $W$ has
no atoms.  Indeed, its law $\mu$ satisfies
$W\stackrel d=\lambda(Y+W')$, where $Y$ is Rademacher and $W'$ is
an independent copy of $W$.  If $\mu$ had atoms, its largest atom mass
$m>0$ would be attained, and the set
$A=\{x:\mu(\{x\})=m\}$ would be finite and nonempty.  For every
$x\in A$,
\[
  m=\frac12\mu\left(\left\{x/\lambda-1\right\}\right)
    +\frac12\mu\left(\left\{x/\lambda+1\right\}\right).
\]
Both masses on the right must equal $m$.  Thus, with
$\alpha=\min A$ and $\beta=\max A$, the points
$\alpha/\lambda-1$ and $\beta/\lambda+1$ belong to $A$ but are
separated by $(\beta-\alpha)/\lambda+2>\beta-\alpha$, a contradiction.

The derivative of $f(x)=|x|^p$ is $(p-1)$-H\"older continuous, so
\begin{equation}
  \left||x+y|^p-|x|^p-p|x|^{p-2}xy\right|
  \le K_p|y|^p,
  \qquad x,y\in\mathbb R,
  \label{eq:p-less-two-remainder}
\end{equation}
where $|x|^{p-2}x$ is interpreted as zero at $x=0$.
For $W\ne0$, the ordinary second-order Taylor expansion and the uniform
bound on $U_r$ imply
\[
  \frac{|W+c_rU_r|^p-|W|^p
        -p|W|^{p-2}Wc_rU_r}{c_r^p}
  =O_W(c_r^{2-p})\longrightarrow0.
\]
The bound \eqref{eq:p-less-two-remainder} dominates these quotients by
$K_p(\lambda/(1-\lambda))^p$.  Since $\Pp(W=0)=0$, dominated
convergence, independence, and $\E U_r=0$ give
\begin{equation}
  \E|W+c_rU_r|^p=M_p+o(c_r^p).
  \label{eq:p-less-two-expansion}
\end{equation}
As $(1-t)^pc_r^p=t$, \eqref{eq:Cr-master} now yields
\begin{equation}
  C_r(H)-M_p=(1-p)tM_p+o(t)<0
  \label{eq:Cr-negative-asymptotic}
\end{equation}
for all sufficiently large $r$.

In either case, fix $r$ large enough that $C_r(H)\ne M_p$ with the
claimed sign.  The split points are dyadic of level $n+r$, so
\begin{equation}
  \mathbb T_n\subseteq\mathbb V_n^{(r)}
  \subseteq\mathbb T_{n+r}.
  \label{eq:small-split-containment}
\end{equation}
It follows that the sequence in \eqref{eq:alternating-partition-main}
is refining and has vanishing mesh.  By
\eqref{eq:small-split-limit} and \eqref{eq:dyadic-pvar}, its odd terminal
sums converge almost surely to $C_r(H)$ and its even terminal sums to
$M_p$.  Their distinct limits prove the theorem.
\end{proof}

At $H=1/2$, the same calculation gives $p=2$, $M_2=1$, and
$\E U_r^2=1-t_r$.  Hence \eqref{eq:Cr-exact} yields $C_r(1/2)=1$
for every $r$, in agreement with Theorem~\ref{thm:main refining}.

\subsubsection{The explicit thirds-grid limit}

\begin{proof}[Proof of Proposition~\ref{prop:thirds-fourth}]
Take the template $\boldsymbol u_\star=\{0,1/3,2/3,1\}$. Then
$\mathbb P_n(\boldsymbol u_\star)=\mathbb U_n$.
Lemma~\ref{lem:finite-template} therefore gives
\begin{equation}
  C_{\mathrm{thirds}}
  =\E\sum_{j=0}^{2}
    \left|\frac13W
      +X^{1/4}_{(j+1)/3}-X^{1/4}_{j/3}\right|^4,
  \label{eq:thirds-constant-representation}
\end{equation}
where $\lambda=2^{-3/4}$ and $W$ is independent of $X^{1/4}$.
Write $b=2^{-1/4}$.  The dyadic orbit of $1/3$ or $2/3$ is always at
relative position $1/3$ or $2/3$ in its tent.  Consequently,
\[
  X^{1/4}_{1/3}=\frac13\sum_{m\ge0}b^m\xi_m,
  \qquad
  X^{1/4}_{2/3}=\frac13\sum_{m\ge0}b^m\zeta_m.
\]
The two points share the level-zero tent, so $\xi_0=\zeta_0$.
At every higher level they lie in distinct tents.  Thus
$\xi_0$ and $\{\xi_m,\zeta_m:m\ge1\}$ are independent Rademacher
variables, also independent of $W$.

Denote the three increments in \eqref{eq:thirds-constant-representation}
by $D_1,D_2,D_3$.  Then
\[
  3D_1=W+\sum_{m\ge0}b^m\xi_m,
  \qquad
  3D_2=W+\sum_{m\ge1}b^m(\zeta_m-\xi_m),
\]
and $D_3\stackrel d=D_1$.  For an absolutely summable coefficient
sequence and independent Rademacher variables $(\eta_i)$,
\begin{equation}
  \E\left(\sum_i a_i\eta_i\right)^4
  =3\left(\sum_i a_i^2\right)^2-2\sum_i a_i^4.
  \label{eq:rademacher-fourth-moment}
\end{equation}
This follows by expanding finite sums and then using bounded convergence.
The coefficient sums for the three series needed here are geometric:
\[
  \begin{array}{c|cc}
    \text{series}&\displaystyle\sum_i a_i^2
                 &\displaystyle\sum_i a_i^4\\[3pt]\hline
    W&\dfrac{1+2\sqrt2}{7}&\dfrac17\\[7pt]
    3D_1&\dfrac{15+9\sqrt2}{7}&\dfrac{15}{7}\\[7pt]
    3D_2&\dfrac{15+16\sqrt2}{7}&\dfrac{15}{7}
  \end{array}
\]
Applying \eqref{eq:rademacher-fourth-moment} gives
\begin{equation}
  C_{\mathrm{dyad}}(1/4)=\E W^4
  =\frac{13+12\sqrt2}{49},
  \label{eq:dyadic-fourth-constant}
\end{equation}
and
\[
  \E D_1^4=\frac{317+270\sqrt2}{1323},
  \qquad
  \E D_2^4=\frac{667+480\sqrt2}{1323}.
\]
Hence
\[
  C_{\mathrm{thirds}}
  =2\E D_1^4+\E D_2^4
  =\frac{1301+1020\sqrt2}{1323},
\]
and subtraction gives the positive difference in
\eqref{eq:different-constants}.
\end{proof}

\subsection{Invariance of variation index}

\subsubsection{Proofs}
\label{sec:proofs-positive-critical}

Throughout this subsection $H\in(0,1)$, $p=1/H$, and we write
\[
  \begin{aligned}
    f_{m,k}&:=2^{m(1/2-H)}e_{m,k},
    &\|f_{m,k}\|_\infty&=2^{-mH-1},\\
    \operatorname{Lip}(f_{m,k})&=2^{m(1-H)},
    &\operatorname{supp}f_{m,k}&=[k2^{-m},(k+1)2^{-m}],
  \end{aligned}
\]
so that $X^H=\sum_{m,k}\theta_{m,k}f_{m,k}$. For an interval
$I=(u,v]$ and a function $g$ we write $\Delta_Ig:=g(v)-g(u)$ and
$|I|:=v-u$.

\begin{proof}[Proof of Proposition~\ref{prop:deterministic-upper}]
By Lemma~\ref{lem:uniform-holder} and $pH=1$,
\[
  [X^H]^{(p)}_{(\tau)}
  =\sum_{I\in\tau}|\Delta_IX^H|^p
  \le L_H^p\sum_{I\in\tau}|I|^{pH}
  =L_H^p\sum_{I\in\tau}|I|=L_H^p .
\]
For $q>p$, $[X^H]^{(q)}_{(\pi_n)}
\le\omega_{X^H}(\mesh\pi_n)^{q-p}[X^H]^{(p)}_{(\pi_n)}
\le L_H^q\mesh(\pi_n)^{(q-p)H}\to0$. Hence
$p^\pi(X^H)\le p$ by Definition~\ref{def. variation index}.
\end{proof}

\begin{lemma}[Second moment of an increment]
\label{lem:increment-second-moment}
There is a constant $c_H>0$ such that, for every interval $I=(u,v]$ with
$0<|I|\le1/2$,
\begin{equation}
  c_H|I|^{2H}\le\E\bigl(\Delta_IX^H\bigr)^2\le L_H^2|I|^{2H}.
  \label{eq:increment-second-moment}
\end{equation}
Consequently there is $c_H'>0$ such that
$\E|\Delta_IX^H|^p\ge c_H'|I|$ for $0<|I|\le1/2$, and
\begin{equation}
  \E\,[X^H]^{(p)}_{(\tau)}\ge c_H'
  \qquad\text{for every partition $\tau$ with }\mesh(\tau)\le1/2.
  \label{eq:critical-mean-lower}
\end{equation}
\end{lemma}

\begin{proof}
The upper bound is Lemma~\ref{lem:uniform-holder}. For the lower bound,
let $h_{m,k}=e_{m,k}'$ be the Haar functions of the proof of
Lemma~\ref{lem:haar-energy}. Since $\Delta_If_{m,k}=2^{m(1/2-H)}\int_uvh_{m,k}$
and the signs are independent with unit variance,
\begin{equation}
  \E\bigl(\Delta_IX^H\bigr)^2
  =\sum_{m\ge0}2^{m(1-2H)}\sum_{k<2^m}
    \Bigl(\int_u^vh_{m,k}\Bigr)^2.
  \label{eq:increment-variance-haar}
\end{equation}
Write $\delta=|I|$ and choose $N\ge0$ with $2^{-N-1}<\delta\le2^{-N}$.

Suppose first that $\delta\le1/16$, so $N\ge4$. The dyadic points of level
at most $N-2$ are spaced $2^{-N+2}\ge4\delta$ apart, so $I$ contains at
most one of them; call it $d$ if it exists, and let $j\le N-2$ be its
level. Fix a level $m\le N-3$. The endpoints and midpoints of the level-$m$
cells are dyadic points of level at most $N-2$. If none of them lies in
$I$, then $I$ is contained in a single level-$m$ cell on one side of its
midpoint, exactly one $k$ contributes to the inner sum in
\eqref{eq:increment-variance-haar}, and its contribution is
$(2^{m/2}\delta)^2=2^m\delta^2$. If $d\in I$ is an endpoint of a
level-$m$ cell, then $I$ meets two cells, each piece lies on one side of
the midpoint of its cell because its length is less than a quarter of the
cell, and the inner sum equals
$2^m\bigl((d-u)^2+(v-d)^2\bigr)\ge2^m\delta^2/2$. If $d$ is the midpoint
of the level-$m$ cell containing $I$, which happens only for $m=j-1$, the
contribution may be smaller and we discard it. Hence
\[
  \E\bigl(\Delta_IX^H\bigr)^2
  \ge\frac{\delta^2}2\Bigl(\sum_{m=0}^{N-3}2^{m(2-2H)}-2^{(N-3)(2-2H)}\Bigr)
  \ge\frac{\delta^2}2\,2^{(N-4)(2-2H)},
\]
where we used that the sum is at least
$2^{(N-3)(2-2H)}(1+2^{-(2-2H)})$. Since $2^N>(2\delta)^{-1}$, this is at
least $c_H\delta^{2H}$ with $c_H=2^{-1-6(2-2H)}$.

For $1/16<\delta\le1/2$, the function $(u,v)\mapsto\E(\Delta_IX^H)^2$
is continuous, because the series \eqref{eq:XH} converges uniformly, and
it is strictly positive: it vanishes only if
$\int_u^vh_{m,k}=0$ for all $(m,k)$, i.e.\ if $\one_{(u,v]}$ is orthogonal
to $H_0$ and hence constant a.e., which forces $\delta\in\{0,1\}$. A
continuous positive function on the compact set
$\{(u,v):1/16\le v-u\le1/2\}$ is bounded below by a positive constant,
and \eqref{eq:increment-second-moment} follows after decreasing $c_H$.
The cutoff $1/2$ can be replaced by any fixed $0<a<1$,
with the lower-bound constant also depending on $a$. A cutoff away from
one is necessary: for $I=(0,1-\varepsilon]$, the bridge property gives
$\E(\Delta_I X^H)^2\le L_H^2\varepsilon^{2H}\to0$, whereas
$|I|^{2H}\to1$.

For the moment bound, write $Z=\Delta_IX^H$, so $|Z|\le L_H\delta^H$. If
$p\ge2$, Jensen's inequality gives
$\E|Z|^p\ge(\E Z^2)^{p/2}\ge c_H^{p/2}\delta^{pH}=c_H^{p/2}\delta$. If
$1<p<2$, then $Z^2=|Z|^p|Z|^{2-p}\le|Z|^p(L_H\delta^H)^{2-p}$, so
$\E|Z|^p\ge\E Z^2\,(L_H\delta^H)^{p-2}\ge c_HL_H^{p-2}\delta^{pH}
=c_HL_H^{p-2}\delta$. Summing over the intervals of $\tau$ gives
\eqref{eq:critical-mean-lower}.
\end{proof}

\begin{lemma}[Bounded differences of the critical sum]
\label{lem:bounded-differences}
Let $\tau$ be any partition of $[0,1]$ with $h:=\mesh(\tau)\le1/2$. Let
$S(\theta):=[X^H]^{(p)}_{(\tau)}$, viewed as a function of the sign array
$\theta=(\theta_{m,k})$, and let $c_{m,k}$ denote the maximal change of
$S$ when $\theta_{m,k}$ alone is flipped. Then
\begin{equation}
  \sum_{m\ge0}\sum_{k<2^m}c_{m,k}^2
  \le C_H\bigl(h^{\kappa_H}+h\one_{\{H=1/2\}}\log(e/h)\bigr),
  \label{eq:bounded-differences-sum}
\end{equation}
with $\kappa_H$ as in \eqref{eq:kappa-H} and a constant $C_H$
depending only on $H$.
\end{lemma}

\begin{proof}
Flipping $\theta_{m,k}$ replaces $X^H$ by $X^H\mp2f_{m,k}$, which is again
a member of the class \eqref{eq:XH}. By the mean value theorem,
$\bigl||a|^p-|b|^p\bigr|\le p\max(|a|,|b|)^{p-1}|a-b|$, and both $a$ and
$b$ are increments over $I$ of functions in the class, so
Lemma~\ref{lem:uniform-holder} and $(p-1)H=1-H$ give
\begin{equation}
  c_{m,k}\le2pL_H^{p-1}\sum_{I\in\tau}|I|^{1-H}\,|\Delta_If_{m,k}|.
  \label{eq:c-mk-bound}
\end{equation}
Put $\delta_I:=|I|$ and $\ell_m:=2^{-m}$. The total variation of
$f_{m,k}$ is $\ell_m^H$, twice its maximum height. Hence
$\sum_{I\in\tau}|\Delta_If_{m,k}|\le\ell_m^H$, and weighted
Cauchy--Schwarz applied to \eqref{eq:c-mk-bound} gives
\begin{equation}
  c_{m,k}^2
  \le C_H\ell_m^H\sum_{I\in\tau}
             \delta_I^{2-2H}|\Delta_If_{m,k}|.
  \label{eq:weighted-differences-critical}
\end{equation}
At each level, the tent derivatives satisfy
$\sum_{k<2^m}|f_{m,k}'(t)|=\ell_m^{H-1}$ for almost every $t$.
Integration over $I$ therefore bounds
$\sum_k|\Delta_If_{m,k}|$ by $\ell_m^{H-1}\delta_I$.
Also, each endpoint belongs to at most one nonzero tent, whose height
is at most $\ell_m^H/2$. The triangle inequality gives the second bound
in
\begin{equation}
  \sum_{k<2^m}|\Delta_If_{m,k}|
  \le\min\{\ell_m^{H-1}\delta_I,\ell_m^H\}.
  \label{eq:tent-increment-bound}
\end{equation}
Combining these estimates and summing nonnegative terms yields
\[
  \sum_{m\ge0}\sum_{k<2^m}c_{m,k}^2
  \le C_H\sum_{I\in\tau}\delta_I^{2-2H}Q_H(\delta_I),
\]
where
\[
  Q_H(\delta)
  :=\delta\sum_{\ell_m\ge\delta}\ell_m^{2H-1}
       +\sum_{\ell_m<\delta}\ell_m^{2H}.
\]
The two geometric sums imply
\[
  Q_H(\delta)\le C_H
  \begin{cases}
    \delta^{2H},&H<1/2,\\
    \delta\log(e/\delta),&H=1/2,\\
    \delta,&H>1/2.
  \end{cases}
\]
Consequently,
\[
  \sum_{m\ge0}\sum_{k<2^m}c_{m,k}^2
  \le C_H
  \begin{cases}
    h,&H<1/2,\\
    h\log(e/h),&H=1/2,\\
    h^{2-2H},&H>1/2.
  \end{cases}
\]
Here we used $\sum_I\delta_I=1$ and $\delta_I\le h$; in the middle
case, $t\mapsto t\log(e/t)$ is increasing on $(0,1]$.
This proves \eqref{eq:bounded-differences-sum} without any restriction
on the relative lengths of the partition intervals.
\end{proof}

\begin{proof}[Proof of Theorem~\ref{thm:critical-concentration}]
Write
$S_n:=[X^H]^{(p)}_{(\pi_n)}$ and $h_n:=\mesh(\pi_n)$. For $n$ large,
$h_n\le1/2$, and \eqref{eq:critical-mean-bounds} follows from
Proposition~\ref{prop:deterministic-upper} and
Lemma~\ref{lem:increment-second-moment}.

For the concentration, let $S_n^{[M]}$ be the same sum computed from the
truncation $X^{[M]}=\sum_{m<M}\sum_k\theta_{m,k}f_{m,k}$.
It is a function of finitely many independent signs.
The proof of Lemma~\ref{lem:uniform-holder} gives the same H\"older
constant $L_H$ when levels are omitted, both before and after a sign
flip. Thus the explicit sensitivity estimates in
\eqref{eq:c-mk-bound}--\eqref{eq:tent-increment-bound} apply to
$S_n^{[M]}$ as well. Summing them as in
Lemma~\ref{lem:bounded-differences} bounds the sum of squared
sensitivities by $C_Hh_n^{\kappa_H}$, uniformly in $M$, when
$H\ne1/2$. McDiarmid's inequality therefore gives, for every
$\varepsilon>0$,
\[
  \Pp\bigl(|S_n^{[M]}-\E S_n^{[M]}|>\varepsilon\bigr)
  \le2\exp\Bigl(-\frac{2\varepsilon^2}{C_H}h_n^{\kappa_H}\Bigr).
\]
As $M\to\infty$, $X^{[M]}\to X^H$ uniformly, hence $S_n^{[M]}\to S_n$ for
every sign array and, by bounded convergence, $\E S_n^{[M]}\to\E S_n$.
Fatou's lemma applied to the indicators of the strict exceedance events
yields
\begin{equation}
  \Pp\bigl(|S_n-\E S_n|>\varepsilon\bigr)
  \le2\exp\Bigl(-\frac{2\varepsilon^2}{C_H}h_n^{\kappa_H}\Bigr).
  \label{eq:critical-mcdiarmid}
\end{equation}
Under \eqref{eq:mesh-condition-critical}, $h_n^{\kappa_H}\log n\to0$, so
the right-hand side is $2n^{-a_n}$ with $a_n\to\infty$, which is summable.
Borel--Cantelli with $\varepsilon=1/j$, intersected over $j\in\mathbb N$,
gives \eqref{eq:critical-asymptotic-determinism}. Combining it with
\eqref{eq:critical-mean-bounds} gives \eqref{eq:critical-liminf-limsup}.

Finally, $p^\pi(X^H)\le p$ by Proposition~\ref{prop:deterministic-upper}.
For $1\le q<p$, on the event \eqref{eq:critical-liminf-limsup},
\[
  [X^H]^{(q)}_{(\pi_n)}
  \ge\frac{[X^H]^{(p)}_{(\pi_n)}}{\omega_{X^H}(h_n)^{p-q}}
  \ge\frac{c_H/2}{L_H^{p-q}h_n^{(p-q)H}}\longrightarrow\infty
\]
for all large $n$, so $p^\pi(X^H)\ge p$ by \eqref{eq : limsup q variation}.
\end{proof}

\begin{proof}[Proof of Corollary~\ref{cor:first-moment}]
(i)$\Rightarrow$(ii) is trivial. If (ii) holds, then $S_n$ converges in
probability and, by \eqref{eq:critical-asymptotic-determinism},
$\E S_n=S_n-(S_n-\E S_n)$ converges in probability; a deterministic
sequence converging in probability converges, which is (iii). If (iii)
holds, \eqref{eq:critical-asymptotic-determinism} gives
$S_n\to\lim_n\E S_n$ almost surely, which is (i) together with the
identification of the limit.
\end{proof}

\begin{remark}
Lemma~\ref{lem:bounded-differences} applies to arbitrary finite
partitions, so Theorem~\ref{thm:critical-concentration} requires neither
balance nor a bound on the number of partition intervals. At $H=1/2$,
the same bounded-differences argument gives almost-sure convergence of
the centered quadratic sums to zero under
\[
  \mesh(\pi_n)\log(e/\mesh(\pi_n))=o(1/\log n).
\]
This sufficient condition is stronger than the one in
Theorem~\ref{thm:sharp-logarithmic-mesh}(a). The Hanson--Wright argument
used there exploits the quadratic structure and removes the extra
logarithmic factor.
\end{remark}

\noindent{\bf Acknowledgements.}
The authors gratefully acknowledge financial support from the Heilbronn Institute for Mathematical Research for an International Visitor Award, which facilitated the research stay of the second author at KCL during which this paper was written. A.S.~ furthermore acknowledges financial support from Munich Re and from the Natural Science
and Engineering Research Council of Canada through grant  RGPIN-2024-03761.

The authors used OpenAI's Codex to assist with reviewing proofs,
identifying possible gaps, and exploring alternative mathematical
arguments. The authors evaluated the suggestions and take full
responsibility for the results, proofs, and final text of the paper.

\bibliographystyle{siam}

\bibliography{ref}

\end{document}